%% file: tau.tex
\documentclass[11pt,a4paper]{gtpart}
\usepackage{hyperref}
\usepackage[textformat=period]{caption}
\usepackage[english,british]{babel}
\usepackage[utf8]{inputenc}
\usepackage{amsmath,amsthm,amssymb,amsfonts}
\usepackage{enumerate}
\usepackage{faktor}
\usepackage{dsfont}
\usepackage{textcomp}
\usepackage{graphicx}
\usepackage{extarrows}
\usepackage{epigraph}
\usepackage{graphicx}
\usepackage{subfigure}
\usepackage{sidecap}
\usepackage{indentfirst}
\usepackage{tikz}
\usepackage{pgfplots}
\usetikzlibrary{arrows,matrix,patterns,decorations.markings,positioning}
\usepackage{caption}
\usepackage{float}
\usepackage{color}
\usepackage{epstopdf}
\usepackage{epsfig}
\usepackage{stmaryrd}
\usepackage{color}

\newcommand{\F}{\mathbb{F}}
\newcommand{\X}{\mathbb{X}}
\newcommand{\OO}{\mathbb{O}}

\newcommand{\lkhov}{\text{\textlbrackdbl}}
\newcommand{\rkhov}{\text{\textrbrackdbl}}

\renewcommand{\geq}{\geqslant}
\renewcommand{\leq}{\leqslant} 

\newtheorem{teo}{Theorem}[section]
\newtheorem*{teo*}{Theorem}
\newtheorem{lemma}[teo]{Lemma}
\newtheorem{prop}[teo]{Proposition}
\newtheorem*{prop*}{Proposition}

\newtheorem{cor}[teo]{Corollary}

\title{The concordance invariant tau in link grid homology}

\author{Alberto Cavallo}
\givenname{Alberto}
\surname{Cavallo}
\email{cavallo\_alberto@phd.ceu.edu}
\address{Alfr\'ed R\'enyi Institute of Mathematics\\
Budapest 1053\\
Hungary}

\begin{document}
%\title{The concordance invariant tau in link grid homology}
%\author{\scshape{Alberto Cavallo}\\ \\
% \footnotesize{Alfr\'ed R\'enyi Institute of Mathematics,}\\
% \footnotesize{Budapest 1053, Hungary}\\ \\ \small{cavallo\_alberto@phd.ceu.edu}}
%\date{}

%\maketitle

\begin{abstract}
 We introduce a generalization of the Ozsv\'ath-Szab\'o $\tau$-invariant to links by studying a filtered 
 version of link grid homology.
 We prove that this invariant remains unchanged under strong concordance and we show that it produces a lower bound
 for the slice genus of a link. We show that this bound is sharp for torus links and we also
 give an application to Legendrian link invariants in the standard contact 3-sphere.
\end{abstract}

\maketitle

\section{Introduction}
Link Floer homology is an invariant for knots and links in three-manifolds, discovered
in 2003 by Ozsv\'ath and Szab\'o \cite{Ozsvath} and independently by Jacob Rasmussen \cite{Rasmussen}, in his PhD thesis.
It is the homology of a
chain complex whose generators are combinatorially defined, and whose differential
counts pseudo-holomorphic disks. 
Grid diagrams are
simple combinatorial presentation of links in $S^3$, dating back to the
$19^{th}$ century. A grid diagram is
an $l\times l$ grid of squares, $l$ of which are marked with an $O$ and 
$l$ of which are marked with an $X$. 
A projection of a link together with
an orientation on it can be associated to a grid diagram $D$.
These grids can be used to give a simpler reformulation of link Floer homology, called
\emph{grid homology}. Of course these two homologies are isomorphic, nevertheless grid homology can be easier to study.

In this paper we use the same notation of the book ``Grid homology for knots and links'' \cite{Book}. In this book
particular attention is given to two versions of the grid homology of a link $L$: the \emph{simply blocked grid homology} 
$\widehat{GH}(L)$ and the \emph{collapsed unblocked grid homology} $cGH^-(L)$; both these homology groups are invariant 
under link equivalence. We study a slightly different version of $\widehat{GH}(L)$. 
Let us denote with $\F$ the field with two elements;
we start constructing a filtered 
$\F$-complex $\left(\widehat{GC}(D),\widehat\partial\right)$
from a grid diagram $D$, equipped with an increasing $\Z$-filtration $\mathcal F$ and we 
prove that $\mathcal F$ induces a filtration in homology, leading to the filtered homology group $\widehat{\mathcal{GH}}(L)$.
The latter is not completely unrelated to $\widehat{GH}(L)$ as we see in Section \ref{section:homology}.  

We can extract a numerical invariant from the homology $\widehat{\mathcal{GH}}(L)$, the integer-valued function 
$T_L:\Z\times\Z\longrightarrow\Z_{\geq 0}$, with the following properties.
\begin{prop}
 \label{prop:one}
 \begin{enumerate}[i)]
  \item The function $T_L$ is supported in $\{1-n,...,0\}\times\Z$ and $\displaystyle\sum_{d,s\in\Z}T_L(d,s)=2^{n-1}$,
        where $n$ is the number of components of $L$.
  \item If $L^*$ is the mirror of the $n$-component link $L$ then 
        $$T_{L^*}(d,s)=T_L(-d+1-n,-s)\:\:\:\:\:\text{for any }d,s\in\Z\:.$$
  \item If $L_1\# L_2$ is a connected sum of $L_1$ and $L_2$ then $T_{L_1\# L_2}$ is the convolution product of
        $T_{L_1}$ and $T_{L_2}$.
  \item If $L$ is a quasi-alternating link then $T_L$ is determined by the signature of $L$.
 \end{enumerate}
\end{prop}
Moreover, in Section \ref{section:cobordism} we prove the following theorem,
which is similar to what Pardon
proved in \cite{Pardon} for Lee homology.
We say that a strong cobordism is a cobordism $\Sigma$, between two links $L_1$ and $L_2$, such that every 
connected component of 
$\Sigma$ is a knot cobordism between a component of $L_1$ and one of $L_2$; in particular $L_1$ and $L_2$ have the same
number of components. If $g(\Sigma)=0$ then $\Sigma$ is a strong concordance.
\begin{teo}
 \label{teo:one}
 The function $T$ is a strong concordance invariant. In other words, if $L_1$ and $L_2$ are
 strongly concordant then
 $T_{L_1}(d,s)=T_{L_2}(d,s)$ for every $d,s\in\Z$.
\end{teo}
In Section \ref{section:invariant} we show that $T_L(0,s)$ is non-zero only for one value of $s$. We call this integer
$\tau(L)$, and, as the name suggests, it
coincides with the classical $\tau$ for knots defined in \cite{Ozsvath}. 
More precisely, we prove the following statement.
\begin{teo}
 \label{teo:set}
 For an $n$-component link the $\tau$-set, defined in \cite{Book} as -1 times the Alexander gradings of a homogeneous,
 free generating set of the torsion-free quotient of $cGH^-(L)$ as an $\F[U]$-module, coincides with
 the $2^{n-1}$ (with multiplicity) values of $s$ 
 where the function $T$ is supported.
\end{teo}
For a knot $K$, where the $\tau$-set is just $\tau(K)$, we have that $T_K(d,s)$ is non-zero only for $(d,s)=(0,\tau(K))$.

From Theorem \ref{teo:one} we know 
that $\tau(L)$ is a strong concordance invariant. Furthermore, it
gives a lower bound for the slice genus $g_4(L)$, that is the 
minimum genus of a compact, oriented, smoothly and properly embedded surface in $D^4$ with $L$ as boundary. 
\begin{prop}
 \label{prop:two}
 For every $n$-component link $L$ we have
 \begin{equation}
  \label{bound}
  |\tau(L)|+1-n\leq g_4(L)\:.
 \end{equation}
\end{prop}
 We use this lower bound to give another proof that, for the positive torus link $T_{q,p}$, we have 
\begin{equation}
 \label{torus_slice}
g_4(T_{q,p})=\frac{(p-1)(q-1)+1-\text{gcd}(q,p)}{2}\:\:\:\:\:\text{for any }q\leq p\:.
\end{equation}
This was already proved by the author in \cite{Cavallo} using the Rasmussen $s$-invariant. 

Finally, we use $\tau(L)$ to prove a generalization of the Thurston-Bennequin number upper bound, given by Olga Plamenevskaya 
in \cite{Plamenevskaya}, to $n$-component Legendrian links.
A brief introduction on Legendrian knots and links can be found in \cite{Geiges}.
\begin{prop}
 \label{prop:three}
 Consider a Legendrian $n$-component link $\mathcal L$ of link type $L$ in $S^3$ equipped with the standard contact structure.
 Then the following inequality holds:
 \begin{equation}
  \label{tb_bound}
  \text{tb}(\mathcal L)+|\text{rot}(\mathcal L)|\leq2\tau(L)-n\:.
 \end{equation}
\end{prop}
Equation \eqref{tb_bound} gives a lower bound for $\tau$ and, using Equation \eqref{bound}, also the following lower 
bound for the slice genus of $L$:
\begin{equation}
 \label{tb_slice}
 \text{tb}(\mathcal L)+|\text{rot}(\mathcal L)|\leq2g_4(L)+n-2\:,
\end{equation}
generalizing a result of Rudolph \cite{Rudolph}
for knots. In Section \ref{section:application} we give an example where this bound is 
sharp. Moreover,
Equation \eqref{tb_bound} can also give an upper bound for $\text{TB}(L)$, the maximal Thurston-Bennequin number of 
a link $L$.
\begin{prop}
 \label{prop:max}
 For every $n$-component link $L$ we have $$\text{TB}(L)\leq 2\tau(L)-n\:.$$
\end{prop}
In particular for a quasi-alternating link, since from iv) in Proposition \ref{prop:one} the invariant
$\tau$ is determined by the 
signature,
we have the following result that Plamenevskaya proved for alternating knots in \cite{Plamenevskaya}.
\begin{cor}
 \label{cor:tb}
 If $L$ is a quasi-alternating link then we have that $$\text{TB}(L)\leq-1-\sigma(L)\:.$$
\end{cor}
As we show in Section \ref{section:application}, the upper bound in Corollary \ref{cor:tb} gives the following proposition.
\begin{figure}[ht]
        \centering
        \def\svgwidth{10cm}
        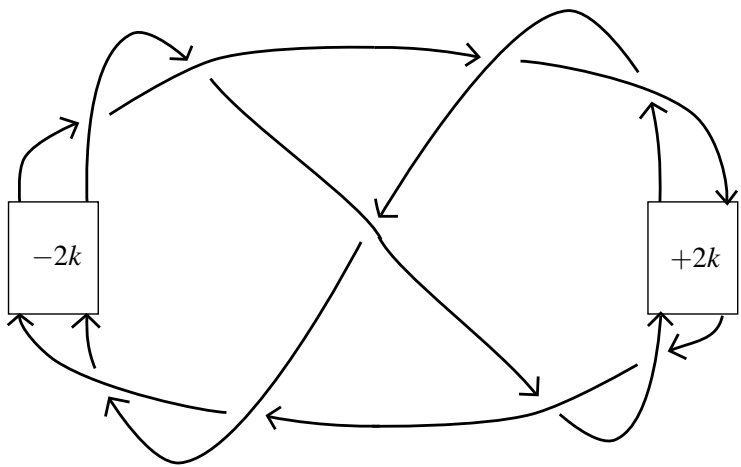   
        \caption{A diagram of $L^k$. For $k=0$ we have the link $L9^a_{40}$}
        \label{L9a40}
\end{figure}
\begin{prop}
 \label{prop:last}
 The links $L^k$ in Figure \ref{L9a40} are a family of two component links, whose components $L^k_i$ are unknots 
 with linking number zero, such that
$\text{TB}(L^k)$ is arbitrarily small.
\end{prop}
The paper is organized as follows. In Section \ref{section:homology} we define the filtered chain complex $\widehat{GC}(D)$
and the homology group $\widehat{\mathcal{GH}}(L)$. In Section \ref{section:invariant} we introduce the function $T_L$
and we prove Proposition \ref{prop:one}. In Section \ref{section:cobordism}
we construct maps in homology, induced by a cobordism $\Sigma$ 
between two links $L_1$ and $L_2$ and we use them to prove 
Theorem \ref{teo:one} and Proposition \ref{prop:two}. In Section \ref{section:minus} we talk briefly
about the filtered version of $cGH^-(L)$ and we explain the proof of Theorem \ref{teo:set}. Finally, 
in Section \ref{section:application} we give some applications, including the
proof of Equations \eqref{torus_slice} and \eqref{tb_bound}.

\subsection*{Acknowledgements}
The author would like to thank Andr\'as Stipsicz for the many helpful conversations and the lots of 
time spent meeting and discussing mathematics.
The author is supported by the ERC Grant LDTBud from the Alfr\'ed R\'enyi Institute of Mathematics and a Full Tuition Waiver
for a Doctoral program at Central European University.

\section{Filtered simply blocked link grid homology}
\label{section:homology}
\subsection{The complex}
We always suppose that a link is oriented.
We denote by $D$ a toroidal grid diagram that represents an $n$-component link $L$. The number $\text{grd}(D)$ is
the number of rows and columns in the grid.
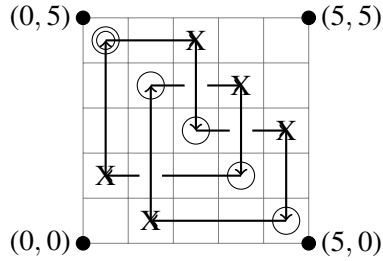
\begin{figure}[ht]
  \begin{center}
   \begin{tikzpicture}[scale=0.6]
    \draw[help lines] (0,0) grid (5,5); 
    \draw (0.5,4.5) circle [radius=0.3];
    \draw (0.5,4.5) circle [radius=0.2]; \node at (2.5,4.5) {X};
    \draw (1.5,3.5) circle [radius=0.3]; \node at (3.5,3.5) {X};
    \draw (2.5,2.5) circle [radius=0.3]; \node at (4.5,2.5) {X};
    \draw (3.5,1.5) circle [radius=0.3]; \node at (0.5,1.5) {X};
    \draw (4.5,0.5) circle [radius=0.3]; \node at (1.5,0.5) {X};
    \draw [thick][->] (0.5,1.5) -- (0.5,4.5); \draw [thick][->] (0.5,4.5) -- (2.5,4.5);
    \draw [thick][->] (1.5,0.5) -- (1.5,3.5); \draw [thick][->] (4.5,0.5) -- (1.5,0.5);
    \draw [thick][->] (2.5,4.5) -- (2.5,2.5); \draw [thick] (1.5,3.5) -- (2.25,3.5); 
                                              \draw [thick][->] (2.75,3.5) -- (3.5,3.5);
    \draw [thick][->] (3.5,3.5) -- (3.5,1.5); \draw [thick] (2.5,2.5) -- (3.25,2.5);
                                              \draw [thick][->] (3.75,2.5) -- (4.5,2.5);
    \draw [thick][->] (4.5,2.5) -- (4.5,0.5); \draw [thick] (3.5,1.5) -- (1.75,1.5);
                                              \draw [thick][->] (1.25,1.5) -- (0.5,1.5);
    \node[left] at (0,0) {$(0,0)$};  \draw[fill] (0,0) circle [radius=0.15];
    \node[left] at (0,5) {$(0,5)$};  \draw[fill] (0,5) circle [radius=0.15];
    \node[right] at (5,0) {$(5,0)$}; \draw[fill] (5,0) circle [radius=0.15];
    \node[right] at (5,5) {$(5,5)$}; \draw[fill] (5,5) circle [radius=0.15];
   \end{tikzpicture}
  \end{center}
  \caption{A grid diagram of the positive trefoil knot}
  \label{Example}
\end{figure}
The orientation in the diagram is taken by going from the $X$ to the $O$-markings in the columns and the opposite in the rows.
Vertical lines are numbered from left to right and horizontal lines from bottom to top, as shown in Figure \ref{Example}.
We identify the boundaries of the grid in order to make it a fundamental domain of a torus; then the lines of the diagram
are embedded circles in this torus. Any $\text{grd}(D)$-tuple of points $x$ in the grid, with
the property that each horizontal and vertical circle contains exactly one of the elements of $x$, 
is called a grid state of $D$.
 
Consider the set of the $O$-markings $\OO=\{O_1,...,O_{\text{grd}(D)}\}$. We call special $O$-markings a non-empty
subset $s\OO\subset\OO$ that contains at most one $O$-marking from each component of $L$, 
while we call the others 
normal $O$-markings. We represent the special
ones with a double circle in the grid diagram. In the paper we usually consider only the case when $\#|s\OO|=n$, which
means there is exactly one special $O$-marking on each component. We talk about the general case in Subsection 
\ref{subsection:new}. From now on there are always $n$ special $O$-markings in a grid diagram, unless it is 
explicitely written differentely.

We define the simply blocked complex $\widehat{GC}(D)$ as the free $\F[V_1,...,V_{\text{grd}(D)-n}]$-module, where 
$\F=\faktor{\Z}{2\Z}$, 
over the grid states $S(D)=\{x_1,...,x_{\text{grd}(D)!}\}$.

We associate to every grid state $x$ the integer $M(x)$, called the Maslov grading of $x$, defined as follows:
\begin{equation}
 \label{maslov}
 M(x)=M_{\OO}(x)=\mathcal J(x-\OO,x-\OO)+1\:;
\end{equation}
where $\mathcal J(P,Q)=\displaystyle\sum_{a\in P}\#\big\{(a,b)\in(P,Q)\:|\:\text{$b$ has both coordinates strictly bigger than 
the}$ $\text{ones of }a\big\}$ with coordinates taken in the interval $[0,\text{grd}(D))$. 

Then we have the Maslov $\F$-splitting  
$$\widehat{GC}(D)=\bigoplus_{d\in\Z}\widehat{GC}_d(D)$$
where $\widehat{GC}_d(D)$ is the finite dimensional $\F$-vector space generated by the elements
$V_1^{l_1}\cdot...\cdot V_{m}^{l_m}x$, with $x\in S(D)$ and $m=\text{grd}(D)-n$, such that 
$$M(V_1^{l_1}\cdot...\cdot V_{m}^{l_m}x)=M(x)-2\sum_{i=1}^ml_i=d\:.$$
We define another integer-valued function on grid states, the Alexander grading $A(x)$, with the formula
$$A(x)=\frac{M(x)-M_{\X}(x)}{2}-\frac{\text{grd}(D)-n}{2}\:;$$
where $M_{\X}(x)$ is defined in Equation \eqref{maslov}, replacing the set $\OO$ with $\X$. 
For the proof that $A(x)$ is really an integer we refer to Chapter 8 in \cite{Book}.

Now we introduce an increasing filtration on $\widehat{GC}(D)$ such that
$$\mathcal F^s\widehat{GC}(D)=\bigoplus_{d\in\Z}\mathcal F^s\widehat{GC}_d(D)$$
and where $\mathcal F^s\widehat{GC}_d(D)$ is generated over $\F$ by the elements $V_1^{l_1}\cdot...\cdot V_{m}^{l_m}x$ with
Maslov grading $d$ and Alexander grading
$$A(V_1^{l_1}\cdot...\cdot V_{m}^{l_m}x)=A(x)-\sum_{i=1}^ml_i\leq s\:.$$

\subsection{The differential}
First we take $x,y\in S(D)$. The set $\text{Rect}(x,y)$ is defined in the following way: it is always empty except when
$x$ and $y$ differs only by a pair of points, say $\{a,b\}$ in $x$ and $\{c,d\}$ in $y$; then $\text{Rect}(x,y)$
consists of the two rectangles in the torus represented by $D$ that have bottom-left and top-right vertices in $\{a,b\}$
and bottom-right and top-left vertices in $\{c,d\}$. We call $\text{Rect}^{\circ}(x,y)\subset\text{Rect}(x,y)$ the subset
of the rectangles which do not contain a point of $x$ (or $y$) in their interior.

The differential $\widehat{\partial}$ is defined as follows:
$$\widehat{\partial}x=\sum_{y\in S(D)}\sum_{\substack{r\in\text{Rect}^{\circ}(x,y) \\ 
r\cap s\OO=\emptyset}}V_1^{O_1(r)}\cdot...\cdot V_m
^{O_m(r)}y\:\:\:\:\text{for any }x\in S(D)$$
where $O_i(r)=\left\{\begin{aligned}
               1\:\:\:\:\:\text{if }O_i\in r \\
               0\:\:\:\:\:\text{if }O_i\notin r \\
                     \end{aligned}\right.$. Here $\{O_1,...,O_m\}$ is the set of the $m=\text{grd}(D)-n$ normal $O$-markings.

We extend $\widehat{\partial}$ to $\widehat{GC}_d(D)$ linearly, and we call it $\widehat{\partial}_d$, 
then again to the whole $\widehat{GC}(D)$ in the following way:
$\widehat{\partial}(V_ix)=V_i\cdot\widehat{\partial}x$ for every $i=1,...,m$ and $x\in S(D)$. Since $\widehat\partial$
keeps the filtration and drops the Maslov grading by 1 (Lemma 13.2.3 in \cite{Book}), we have maps
$$\widehat{\partial}_{d,s}:\mathcal F^s\widehat{GC}_d(D)\longrightarrow\mathcal F^s\widehat{GC}_{d-1}(D)$$
where $\widehat{\partial}_{d,s}$ is the restriction of $\widehat{\partial}_d$ to the subspace 
$\mathcal F^s\widehat{GC}_d(D)\subset\widehat{GC}_d(D)$. 
Furthermore $\widehat{\partial}\circ\widehat{\partial}=0$ (Lemma 13.2.2 in \cite{Book}).

\subsection{The homology}
\label{subsection:homology}
We define the homology group $\widehat{\mathcal{GH}}_d(D)$ as the quotient space
$\dfrac{\text{Ker }\widehat{\partial}_d}{\text{Im }\widehat{\partial}_{d+1}}$.
Moreover, we introduce the subspaces $\mathcal F^s\widehat{\mathcal{GH}}_d(D)$ as follows: consider the projection
$\pi_d:\text{Ker }\widehat\partial_d\rightarrow\widehat{\mathcal{GH}}_d(D)$. Since
$\text{Ker }\widehat{\partial}_{d,s}=\text{Ker }\widehat{\partial}_{d}\cap\mathcal F^s\widehat{GC}_d(D)$
we say that
$$\mathcal F^s\widehat{\mathcal{GH}}_d(D)=\pi_d(\text{Ker }\widehat{\partial}_{d,s})$$
for every $s\in\Z$. 
$\text{Ker }\widehat{\partial}_{d,s}\subset\text{Ker }\widehat{\partial}_{d,s+1}$ 
implies that the filtration $\mathcal F$ descends to homology. We see immediately that each 
$\mathcal F^s\widehat{\mathcal{GH}}_d(D)$ is a finite dimensional $\F$-vector space.

We can extend the filtration $\mathcal F$ on the total homology 
$$\widehat{\mathcal{GH}}(D)=\bigoplus_{d\in\Z}\widehat{\mathcal{GH}}_d(D)$$
by taking
$$\mathcal F^s\widehat{\mathcal{GH}}(D)=\bigoplus_{d\in\Z}\mathcal F^s\widehat{\mathcal{GH}}_d(D)\:.$$

From \cite{Book} Chapter 13 we know that the dimension of
$\mathcal F^s\widehat{\mathcal{GH}}_d(D)$ as an $\F$-vector space is a link 
invariant for every $d,s\in\Z$, 
in particular they are independent of the choice of the special $O$-markings and the ordering of the markings.
Hence we can denote them with $\mathcal F^s\widehat{\mathcal{GH}}_d(L)$. Furthermore, \cite{Book} Lemma 13.2.5
says that $[V_ip]=[0]$ for every $i=1,...,m$ and $[p]\in \widehat{\mathcal{GH}}(L)$. This means that each
homology class can be represented by a combination of grid states and every level 
$\mathcal F^s\widehat{\mathcal{GH}}(L)$ is also a finite dimensional $\F$-vector space. 

The homology group $\widehat{GH}_{d,s}(L)$ of \cite{Book} can be recovered from the complex $\widehat{GC}(D)$ in the
following way.
We denote the graded object associated to a filtered complex
$\mathcal C$ the bigraded chain complex 
$\big(\text{gr}(\mathcal C),\text{gr}(\partial)\big)$,
where
$$\text{gr}(\mathcal C)_{d,s}=\frac{\mathcal F^s\mathcal C_d}{\mathcal F^{s-1}\mathcal C_d}$$
and $\text{gr}(\partial)$ is the map induced by $\partial$ on $\text{gr}(\mathcal C)$.
Then we have that 
$$\widehat{GH}_{d,s}(L)\cong_{\F}H_{d,s}\left(\text{gr}\left(\widehat{GC}(D)\right),\text{gr}(\widehat\partial)\right)\:.$$

\section{The invariant tau in the filtered theory}
\label{section:invariant}
\subsection{Definitions}
\label{subsection:definitions}
Since $\mathcal F^{s-1}\widehat{\mathcal{GH}}_d(L)\subset\mathcal F^s\widehat{\mathcal{GH}}_d(L)$, 
and they are finite dimensional vector spaces,
we define the function
$$T_L(d,s)=\text{dim}_{\F}\frac{\mathcal F^s\widehat{\mathcal{GH}}_d(L)}{\mathcal F^{s-1}\widehat{\mathcal{GH}}_d(L)}$$
which clearly is still a link invariant.

Our first goal is to see what happens to this function $T$ when we stabilize the link $L$, in other words
when we add a disjoint 
unknot to $L$. Denote the unknot with the symbol $\bigcirc$. We claim that 
\begin{equation}
 \label{unknot}
 T_{L\sqcup\bigcirc}(d,s)=T_L(d,s)+T_L(d+1,s)\:\:\:\:\:\text{for any }d,s\in\Z\:.
\end{equation}
Before the proof of Equation \eqref{unknot} it is time for some remarks on filtered chain maps.

Suppose $f:(\mathcal C,\partial)\rightarrow(\mathcal C',\partial')$
is a chain map between two filtered chain complexes over $\F$. We say that
$f$ is filtered of degree $t$ if $f(\mathcal F^s\mathcal C)\subset\mathcal F^{s+t}\mathcal C'$ for every $s\in\Z$.

A filtered chain map induces a map in homology that is filtered of the same degree. This means that $f$ induces a map
$f_*:H_*(\mathcal C)\rightarrow H_*(\mathcal C')$ such that $f_*(\mathcal F^s H_*(\mathcal C))\subset\mathcal F^{s+t}
H_*(\mathcal C')$ for any $s\in\Z$.

We say that a linear map $F:H_*(\mathcal C)\rightarrow H_*(\mathcal C')$
is a filtered isomorphism if $F$ is 
bijective and $F(\mathcal F^s H_*(\mathcal C))=\mathcal F^{s}H_*(\mathcal C')$ for any $s\in\Z$.
We denote with $H_*(\mathcal C)\cong H_*(\mathcal C')$ two filtered isomorphic homology groups such that the
isomorphism preserves the grading; more excplicitely this means that 
$\mathcal F^s H_d(\mathcal C)\cong_{\F}\mathcal F^s H_d(\mathcal C')$ for every $d,s\in\Z$.

Moreover, we can associate to a filtered chain map $f:\mathcal C\rightarrow\mathcal C'$ the quotient map
$$\text{gr}(f):\text{gr}(\mathcal C)\longrightarrow
\text{gr}(\mathcal C')\:.$$
We call $f$ a filtered quasi-isomorphism if the map $\text{gr}(f)$ 
induces an isomorphism between $H_{*,*}(\text{gr}(\mathcal C))$ and $H_{*,*}(\text{gr}(\mathcal C'))$ that preserves
the gradings. We denote with $\mathcal C\cong\mathcal C'$ two filtered quasi-isomorphic
complexes.

From Proposition A.6.1 in \cite{Book} we have that if there is a filtered quasi-isomorphism between 
$(\mathcal C,\partial)$ and $(\mathcal C',\partial')$ then
$H_*(\mathcal C)\cong H_*(\mathcal C')$. While from Proposition A.8.1 in \cite{Book}
we know that $\mathcal C\cong\mathcal C'$ if and only if there is a filtered chain homotopy
equivalence between $\mathcal C$ and $\mathcal C'$, provided $\mathcal C$ and $\mathcal C'$
are also modules over $\F[V_1,...,V_m]$. See \cite{Book} Chapter 13 and Appendix A for more details.

Finally, we define the 
shifted complex $\mathcal C\lkhov a,b\rkhov=\mathcal C'$ as $\mathcal F^s\mathcal C'_d=
\mathcal F^{s-b}\mathcal C_{d-a}$. Now, in order to prove Equation \eqref{unknot}, we need the following proposition.
\begin{prop}
 \label{prop:unknot}
 For any link $L$ we have $\widehat{\mathcal{GH}}(L\sqcup\bigcirc)\cong\widehat{\mathcal{GH}}(L)\otimes V$,
 where $V$ is the two dimensional $\F$-vector space with generators in grading and minimal 
 level $(d,s)=(-1,0)$ and $(d,s)=(0,0)$.
\end{prop}
\begin{proof} 
 Take a grid diagram $D$ for $L$. Then the extended
 diagram $\overline{D}$, obtained from $D$ by adding one column on the left and
 one row on the top with a doubly-marked square in the top left, represents the link $L\sqcup\bigcirc$.
 The circle in the doubly-marked square is forced to be a special $O$-marking and we can also suppose that there is 
 another special $O$-marking just below and right of it, as shown in Figure \ref{Stabilization}.
 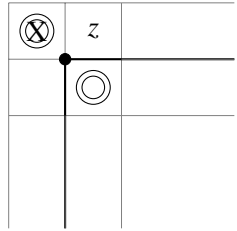
\begin{figure}[ht]
  \begin{center}
   \begin{tikzpicture}[scale=0.75]
    \draw[help lines] (0,0) grid (2,2); 
    \draw[thick] (1,1)--(4,1);       \draw[thick] (1,1)--(1,-2);
    \draw[help lines] (0,-2)--(0,0); \draw[help lines] (2,0)--(4,0);
    \draw[help lines] (1,-2)--(1,0); \draw[help lines] (2,1)--(4,1);
    \draw[help lines] (2,-2)--(2,0); \draw[help lines] (2,2)--(4,2);
    \draw[fill] (1,1) circle [radius=0.1];
    \draw (0.5,1.5) circle [radius=0.3]; \draw (1.5,0.5) circle [radius=0.2];
    \draw (0.5,1.5) circle [radius=0.2]; \draw (1.5,0.5) circle [radius=0.3];
    \node at (0.5,1.5) {X}; \node at (1.5,1.5) {$z$};
   \end{tikzpicture}
  \end{center}
  \caption{We call the top-left special $O$-marking $O_0$}
  \label{Stabilization}
 \end{figure} 
 
 Let $I(\overline D)$ denote the set of those generators in $S(\overline D)$ which have a component at the
 lower-right corner of the doubly-marked square, and let $N(\overline D)$ be the complement
 of $I(\overline D)$ in $S(\overline D)$. From the placement of the special $O$-markings, we see that $N(\overline D)$ spans
 a subcomplex $\mathbf N$ in $\widehat{GC}(\overline D)$. Moreover, if $\widehat\partial_1$ is the differential in
 $\widehat{GC}(\overline D)$ and $\widehat\partial_2$ the one in $\widehat{GC}(D)$
 we can express the restriction of
 $\widehat\partial_1$ to
 the subspace $\mathbf I$, spanned by $I(\overline D)$, with $\widehat\partial_2+\widehat\partial_{\mathbf N}$; this
 because there is a one-to-one correspondence between elements of $I(\overline D)$ and grid states
 in $S(D)$. This correspondence induces
 a filtered quasi-isomorphism $i:(\mathbf I,\widehat\partial_2)\rightarrow\widehat{GC}(D)$.
 
 Define a map $H:\mathbf N\rightarrow\mathbf I$ by the formula
 $$H(x)=\sum_{y\in I(\overline D)}\sum_{\substack{r\in\text{Rect}^{\circ}(x,y) \\ 
 O_0\in r}}V_1^{O_1(r)}\cdot...\cdot V_m^{O_m(r)}y\:\:\:\:\:\text{for any }x\in N(\overline D)\:.$$
 We have that $H$ is a filtered chain homotopy equivalence between 
 $(\mathbf N,\widehat\partial_1)$ and $(\mathbf I,\widehat\partial_2)$, which increases the Maslov grading by one, and
 $H\circ\widehat\partial_{\mathbf N}=0$.
 To see the first claim, we mark the square just on the right of $O_0$ with $z$ and we define an operator 
 $\widehat H_z:\mathbf I\rightarrow\mathbf N$, which counts only rectangles that contain $z$. This operator is a homology
 inverse of $H$; this and the second claim can be proved in the same way as in Lemma 8.4.7 in \cite{Book}.
 Those two facts together tell us that the following diagram commutes. 
 \begin{figure}[H]
  \begin{center}
   \begin{tikzpicture}[node distance=3cm, auto]
                       \node (A) {$\mathbf I$};
		       \node (B) [right of=A] {$\mathbf N$};
                       \node (C) [below of=A,yshift=1cm] {$\widehat{GC}(D)$};
                       \node (D) [right of=C] {$\widehat{GC}(D)\lkhov-1,0\rkhov$};
                       \draw[->] (A) to node  {$\widehat\partial_{\mathbf N}$} (B);
                       \draw[->] (A) to node  [swap]{$i$} (C);
                       \draw[->] (B) to node  {$i\circ H$} (D);
                       \draw[->] (C) to node  {$0$} (D);
   \end{tikzpicture}
  \end{center}
 \end{figure}
 Since $H$ is a filtered chain homotopy equivalence, $i\circ H$ is a filtered quasi-isomorphism, just like the map $i$. 
 Therefore, we can use a filtered version of Lemma A.3.8 in \cite{Book} and 
 obtain that the map between the mapping cones
 $$\text{Cone}(\widehat\partial_{\mathbf N})=\widehat{GC}(\overline D)\longrightarrow\text{Cone}(0)=\widehat{GC}(D)\otimes V$$
 is a filtered quasi-isomorphism and so the claim follows easily. See also the proof of Lemma 8.4.7 in \cite{Book} for other
 details.
\end{proof}
Equation \eqref{unknot} is obtained immediately from Proposition \ref{prop:unknot}, in fact we have proved that
$\mathcal F^s\widehat{\mathcal{GH}}_d(L\sqcup\bigcirc)\cong\mathcal F^s\widehat{\mathcal{GH}}_d(L)\oplus
\mathcal F^s\widehat{\mathcal{GH}}_{d+1}(L)$ for every $d,s\in\Z$ and so it is enough to apply the definition of $T$.

Now we are able to do some computations. 
The homology of the unknot can be easily computed by taking the grid diagram of dimension 1, where the square is marked 
with both $X$ and $O$. The complex has one element of Maslov and Alexander grading 0; then $T_{\bigcirc}(d,s)=1$ if
$(d,s)=(0,0)$ and it is 0 otherwise.

Using Equation \eqref{unknot} we get the function $T$ of the $n$-component unlink $\bigcirc_n$:
$$T_{\bigcirc_n}(d,s)=\left\{\begin{aligned}
                                 &\binom{n-1}{k}\:\:\:\:\:\text{if }(d,s)=(-k,0),\:\:\:0\leq k\leq n-1 \\
                                 &\:\:\:\:\:\:\:0\hspace{3cm}\text{otherwise}
                                \end{aligned}\right.\:.$$     
We can also see this directly from Proposition \ref{prop:unknot}, 
in fact we have that $\widehat{\mathcal{GH}}(\bigcirc_n)\cong V^{\otimes(n-1)}$.

Now let us consider a grid diagram $D$ of a link $L$. The Maslov grading of the elements of $S(D)$ 
and the differential $\widehat{\partial}$ are independent of the position
of the $X$'s, once we have fixed the special $O$-markings. Since we can always change the $X$-markings to obtain 
$\bigcirc_n$, this means that 
$\text{dim}_{\F}\widehat{\mathcal{GH}}_d(L)=\text{dim}_{\F}\widehat{\mathcal{GH}}_d(\bigcirc_n)$
for every $d\in\Z$ and the generators are the same. In particular
$$\widehat{\mathcal{GH}}(L)\cong_{\F}\widehat{\mathcal{GH}}(\bigcirc_n)\cong_{\F}\F^{2^{n-1}}\:\:\:\:\:\text{and}\:\:\:\:\:
\widehat{\mathcal{GH}}_d(L)\cong_{\F}\F^{\binom{n-1}{-d}}\:\:\:\:\:\text{when }1-n\leq d\leq 0\:.$$

From this we have that $\widehat{\mathcal{GH}}_0(L)$ has always dimension 1 and then we define $\tau(L)$ as
the only integer $s$ such that $T_L(0,s)>0$, as we previously said
in the Introduction. We remark that
for a knot this version of $\tau$ coincides with the one of Ozsv\'ath and Szab\'o. See the proof of Theorem \ref{teo:set}
in Section \ref{section:minus}. We also observe that
Equation \eqref{unknot} tells that $\tau(L\sqcup\bigcirc)=\tau(L)$.

\subsection{Dropping the special \texorpdfstring{$O$}{Lg}-markings}
\label{subsection:new}
In this subsection we study what happens to the homology of the filtered chain complex 
$\left(\widehat{GC}(D),\widehat\partial\right)$ if the grid diagram $D$ has less than $n$ special $O$-markings.

Let us consider $D$ a grid diagram for an $n$-component link $L$. The set $s\OO\subset\OO$ contains at most one $O$-marking
from each component of $L$, but we have that $\#|s\OO|\geq 1$. Denote with $m=\text{grd}(D)-\#|s\OO|$ the number of normal
$O$-markings in $D$. Then we can define our chain complex exactly 
in the same way as in Section \ref{section:homology}; on the other
hand, the homology $\widehat{\mathcal{GH}}(D)$ is no longer a link invariant, in fact it clearly depends of the choice of
the special $O$-markings.

Nonetheless, we can show that the $\F$-vector space $\widehat{\mathcal{GH}}(D)$ is still finite dimensional.
Note that this is not true if instead we consider the simply blocked homology group $\widehat{GH}(D)$.
\begin{prop}
 The homology group $\widehat{\mathcal{GH}}(D)$, defined as in Subsection \ref{subsection:homology}, is $\F$-isomorphic to
 $\widehat{\mathcal{GH}}\left(\bigcirc_{\#|s\OO|}\right)$, 
 the homology of the unlink with $\#|s\OO|$ components each containing a special $O$-marking. 
 In particular, we have that
 $$\widehat{\mathcal{GH}}(D)\cong_{\F}\F^{2^{\#|s\OO|}-1}\:\:\:\:\:\text{and}\:\:\:\:\:
 \widehat{\mathcal{GH}}_d(D)\cong_{\F}\F^{\binom{\#|s\OO|-1}{-d}}$$ when $1-\#|s\OO|\leq d\leq 0$. 
\end{prop}
\begin{proof}
 As we noted before, the group $\widehat{\mathcal{GH}}(D)$ does not depend on the position of the $X$-marking. Since we can
 always change them in a way that $D$ becomes a diagram for an unlink with a special $O$-marking on every component, the  
 claim follows from the results in the previous subsection.
\end{proof}
Even though in this case the homology is no longer a link invariant, we can still prove the following theorem.
\begin{teo}
 \label{teo:new}
 Let us consider two grid diagrams $D_1$ and $D_2$ representing smoothly isotopic links $L_1$ and $L_2$ such that all the 
 isotopic components both contain or not contain a special $O$-marking.
 Then we have that $\widehat{\mathcal{GH}}(D_1)$ is filtered isomorphic to $\widehat{\mathcal{GH}}(D_2)$ and the 
 isomorphism preserves the Maslov grading. Hence, we can denote
 the homology group of an $n$-component link $L$ with $\widehat{\mathcal{GH}}^{\OO}(L)$ and it depends only on which 
 components of the link contain a special $O$-marking.
\end{teo}
\begin{proof}
 Lemma 4.1 in \cite{Sarkar} tells us that such two grid diagrams differ by a finite sequence of grid moves: reordering
 of the $O$-markings, commutations and stabilizations. See Section 3 in \cite{Sarkar} for more details. Then it is enough
 to prove the theorem in the case when $D_2$ is obtained from $D_1$ by one of these three moves.
 
 Applying the results in \cite{Book} Chapter 5 and 13 we find filtered quasi-isomorphisms for each move and this implies
 $\widehat{\mathcal{GH}}(D_1)\cong\widehat{\mathcal{GH}}(D_2)$.
\end{proof}
We use the homology groups $\widehat{\mathcal{GH}}^{\OO}(L)$ to define some cobordism maps in Section \ref{section:cobordism}.

\subsection{Symmetries}
\subsubsection{Reversing the orientation}
If $-L$ is the link obtained from $L$ by reversing the orientation of all the components then
\begin{equation}
 \label{reverse}
 T_{-L}(d,s)=T_L(d,s)\:\:\:\:\:\text{for any }d,s\in\Z
\end{equation} 
and $\tau(-L)=\tau(L)$.

To see this, consider a grid diagram $D$ of $L$, then it is easy to observe that, if we reflect $D$ along the diagonal
going from the top-left to the bottom-right of the grid, the diagram $D'$ obtained in this way represents $-L$.
Hence, we take the map $\Phi:S(D)\rightarrow S(D')$ that sends a grid state $x$ into its reflection $x^-$ and now,
from Proposition 4.3.1 in \cite{Book}, we have that $M(x^-)=M(x)$ and $A(x^-)=A(x)$. This means that $\Phi$ is a 
filtered quasi-isomorphism between $\widehat{GC}(D)$ and $\widehat{GC}(D')$, since clearly the differentials commute with 
$\Phi$.
This gives that $\widehat{\mathcal{GH}}(-L)\cong\widehat{\mathcal{GH}}(L)$ and then Equation \eqref{reverse}
follows.

\subsubsection{Mirror image}
For an $n$-component link $L$ we have that the function $T$ of the mirror image $L^*$ is given by the following
equation
\begin{equation}
 \label{mirror}
 T_{L^*}(d,s)=T_L(-d+1-n,-s)\:\:\:\:\:\text{for any }d,s\in\Z\:.
\end{equation}
The proof of this relation is similar to the proof of 
Proposition 7.1.2 in \cite{Book}. First, given a complex $\mathcal C$ with a
filtration $\mathcal F$, we introduce the filtered dual complex $\mathcal C^*$, 
equipped with a filtration 
$\mathcal F^*$ by taking
$$(\mathcal F^*)^s(\mathcal C^*)_d=\text{Ann}(\mathcal F^{-s-1}\mathcal C_{-d})\subset(\mathcal C_{-d})^*=(\mathcal C^*)_d
\:\:\:\:\:\text{for any }d,s\in\Z\:,$$
where $\text{Ann}(\mathcal F^h\mathcal C_k)$ is the subspace of $(\mathcal C_k)^*$ consisting of 
all the linear functionals
that are zero over $\mathcal F^h\mathcal C_k$.

Second, given a grid diagram $D$ of $L$, 
we call $\left(\widetilde{GC}(D),\widetilde\partial\right)$ the filtered fully blocked chain complex 
$\dfrac{\left(\widehat{GC}(D),\widehat\partial\right)}{V_1=...=V_m=0}$ and
we also denote with $W$ the two dimensional $\F$-vector space with generators in grading and minimal 
level $(d,s)=(0,0)$ and $(d,s)=(-1,-1)$.
We want to prove the following proposition.
\begin{prop}
 \label{prop:mirror}
 $\widehat{\mathcal{GH}}(L^*)\cong\widehat{\mathcal{GH}}^*(L)\llbracket1-n,0\rrbracket$, where the filtration
 on $\widehat{\mathcal{GH}}^*(L)$ is $\mathcal F^*$.
\end{prop}
\begin{proof}
 Let $D^*$
 be the diagram obtained by reflecting $D$ through a
 horizontal axis. The diagram $D^*$ represents $L^*$. Reflection induces a bijection
 $x\rightarrow x^*$ between grid states for $D$ and those for $D^*$, inducing a bijection between
 empty rectangles in $\text{Rect}^{\circ}(x,y)$ and empty rectangles in $\text{Rect}^{\circ}(y^*,x^*)$. Hence, the
 reflection induces a filtered isomorphism 
 $$\widetilde{GC}(D^*)\cong\widetilde{GC}^*(D)\lkhov1-\text{grd}(D),n-\text{grd}(D)
 \rkhov\:,$$
 where the shifts are given by the fact that $M(x^*)=-M(x)+1-\text{grd}(D)$ and $A(x^*)=-A(x)+n-\text{grd}(D)$.
 
 Now, from Lemma 14.1.11 in \cite{Book}, we have the filtered quasi-isomorphism
 \begin{equation}
  \label{tilde}
  \widetilde{GC}(D)\cong\widehat{GC}(D)\otimes W^{\otimes(\text{grd}(D)-n)}\:.
 \end{equation} 
 Combining these two relations and observing that $$(W^*)^{\otimes(\text{grd}(D)-n)}
 \cong W^{\otimes(\text{grd}(D)-n)}\lkhov\text{grd}(D)-n,\text{grd}(D)-n\rkhov$$ leads to the
 filtered quasi-isomorphism
 $$\widehat{GC}(D^*)\cong\widehat{GC}^*(D)\lkhov1-n,0\rkhov\:.$$ 
\end{proof}
Proposition \ref{prop:mirror} says that 
$\mathcal F^s\widehat{\mathcal{GH}}_d(L^*)\cong(\mathcal F^*)^s\left(\widehat{\mathcal{GH}}^*\right)_{d-1+n}(L)$ 
for every $d,s\in\Z$. Then we can prove Equation \eqref{mirror}:
$$\begin{aligned}
        T_{L^*}(d,s)&=\text{dim }\frac{(\mathcal F^*)^s\left(\widehat{\mathcal{GH}}^*\right)_{d-1+n}(L)}
 {(\mathcal F^*)^{s-1}\left(\widehat{\mathcal{GH}}^*\right)_{d-1+n}(L)}=
 \text{dim }\frac{\text{Ann}\left(\mathcal F^{-s-1}\widehat{\mathcal{GH}}_{-d+1-n}(L)\right)}
 {\text{Ann}\left(\mathcal F^{-s}\widehat{\mathcal{GH}}_{-d+1-n}(L)\right)}=\\
 &=\text{dim }\frac{\mathcal F^{-s}\widehat{\mathcal{GH}}_{-d+1-n}(L)}
 {\mathcal F^{-s-1}\widehat{\mathcal{GH}}_{-d+1-n}(L)}
 =T_L(-d+1-n,-s)\:.
\end{aligned}$$
If we define $\tau^*(L)$ as the unique integer such that $T_L(1-n,\tau^*(L))=1$ then we have proved
that 
\begin{equation}
 \label{tau_bar}
 \tau(L^*)=-\tau^*(L)\:.
\end{equation}
In particular for a knot $K$, where $\tau^*(K)=\tau(K)$, we have
$\tau(K^*)=-\tau(K)$.
Moreover, we have the following corollary.
\begin{cor}
 \label{cor:mirror}
 Suppose that $L$ is smoothly isotopic to $L^*$. 
 Then, $$T_{L^*}(d,s)=T_L(d,s)\:\:\:\:\:\text{for any }d,s\in\Z$$ and so Equation \eqref{mirror} gives that
 the function $T_L$ has a central symmetry in the point $\left(\frac{1-n}{2},0\right)$.
 In particular $\tau^*(L)=-\tau(L)$ and, for knots, $\tau(K)=0$.
\end{cor}

\subsubsection{Connected sum}
Given two links $L_1$ and $L_2$, the function $T$ of the connected sum $L_1\# L_2$ is the convolution product of the
$T$ functions of $L_1$ and $L_2$; in other words
\begin{equation}
 \label{connected_sum}
 T_{L_1\# L_2}(d,s)=\sum_{\substack{d=d_1+d_2\\ s=s_1+s_2}}T_{L_1}(d_1,s_1)\cdot T_{L_2}(d_2,s_2)\:\:\:\:\:
 \text{for any }d,s\in\Z\:.
\end{equation}
This equation is very hard to prove in the grid diagram settings, but it has been proved quite easily using the 
holomorphic definition of link Floer Homology.
In fact Ozsv\'ath and Szab\'o proved that, 
if $D$ is a Heegaard diagrams for $L$ and $\widehat{CFL}(D)$ denotes the link Floer complex, 
there is a filtered chain homotopy equivalence
$$\widehat{CFL}(D_1)\otimes\widehat{CFL}(D_2)\longrightarrow\widehat{CFL}(D_1\# D_2)$$
which gives that $\widehat{\mathcal{GH}}(L_1\# L_2)\cong\widehat{\mathcal{GH}}(L_1)\otimes\widehat{\mathcal{GH}}(L_2)$.
See Section 7 in \cite{Ozsvath}.

We see immediately that the homology and the
$T$ function of $L_1\# L_2$ are independent from the choice of the components used to perform the connected sum;
moreover the $\tau$-invariant is additive: $$\tau(L_1\# L_2)=\tau(L_1)+\tau(L_2)\:.$$

\subsubsection{Disjoint union}
The disjoint union of two links $L_1$ and $L_2$ is equivalent to $L_1\#(L_2\sqcup\bigcirc)$. Thus by Equation
\eqref{unknot} and \eqref{connected_sum} we have the following relation:
\begin{equation}
 \label{union}
 T_{L_1\sqcup L_2}(d,s)=\sum_{\substack{d=d_1+d_2\\ s=s_1+s_2}}T_{L_1}(d_1,s_1)\cdot\big(T_{L_2}(d_2,s_2)+
 T_{L_2}(d_2+1,s_2)\big)\:\:\:\:\:\text{for any }d,s\in\Z\:,
\end{equation}
or in other words:
$\widehat{\mathcal{GH}}(L_1\sqcup L_2)\cong\widehat{\mathcal{GH}}(L_1)\otimes\widehat{\mathcal{GH}}(L_2)\otimes V$,
where $V$ is the two dimensional $\F$-vector space with generators in grading and minimal 
level $(d,s)=(-1,0)$ and $(d,s)=(0,0)$. We have immediately that 
$$\tau(L_1\sqcup L_2)=\tau(L_1\# L_2)=\tau(L_1)+\tau(L_2)\:.$$

\subsubsection{Quasi-alternating links}
We recall that quasi-alternating links are the smallest set of links $\mathcal Q$ that satisfies the two properties:
\begin{enumerate}
 \item The unknot is in $\mathcal Q$.
 \item $L$ is in $\mathcal Q$ if it admits a diagram with a crossing whose two resolutions
       $L_0$ and $L_1$ are both in $\mathcal Q$,
       $\text{det}(L_i)\neq 0$ and $\text{det}(L_0)+\text{det}(L_1)=\text{det}(L)$.
\end{enumerate}
The above definition and Lemma 3.2 in \cite{Ozsvath1} imply that
every quasi-alternating link is non-split and every non-split alternating link is quasi-alternating.
Moreover, quasi-alternating links are both Khovanov and link Floer homology thin, which means that their homologies
are supported in two and one lines respectively, and the homology is completely determined by the signature and the Jones 
(Alexander in the hat version of link Floer homology)
polynomial. The following proposition says that the same is true in filtered grid homology.
\begin{teo}
 \label{teo:quasi}
 If $L$ is an $n$-component quasi-alternating link then the function $T_L$ is supported in a line; more
 specifically the following relation holds:
 $$T_L(d,s)\neq0\:\:\:\:\:\text{if and only if}\:\:\:\:\:
 s=d+\frac{n-1-\sigma(L)}{2}\:\:\:\:\:\text{for }1-n\leq d\leq0$$
 where $\sigma(L)$ is the signature of $L$.
\end{teo}
\begin{proof}
 We already know that if $T_L(d,s)\neq0$ then $1-n\leq d\leq 0$, 
 so we only have to prove the alignment part of the statement.
 
 Take a grid diagram $D$ for $L$, then from  Theorem 10.3.3 in \cite{Book} the claim is true for the bigraded homology
 $H_{*,*}\left(\text{gr}\left(\widehat{GC}(D)\right)\right)\cong\widehat{GH}(L)$.
 Since $T_L(d,s)\neq0$ implies that $H_{d,s}\left(\text{gr}\left(\widehat{GC}(D)\right)\right)$ is non zero, the theorem 
 follows.
\end{proof}
From Theorem \ref{teo:quasi} we obtain immediately the following
corollary.  
\begin{cor}
 \label{cor:quasi}
 If $L$ is an $n$-component quasi-alternating link then
 $\tau(L)=\frac{n-1-\sigma(L)}{2}$ and $\tau(L^*)=n-1-\tau(L)$.
\end{cor}

\section{Cobordisms}
\label{section:cobordism}
\subsection{Induced maps and degree shift}
\label{subsection:shift}
In this section we study the behaviour of the function $T$ under cobordisms. 
A genus $g$ cobordism between two links $L_1$ and $L_2$ 
is a smooth embedding 
$f:\Sigma_{g}\rightarrow S^3\times I$ where $\Sigma_{g}$ is a compact orientable surface of genus $g$ 
(more precisely $\Sigma_g$ has connected components $\Sigma_{g_1},...,\Sigma_{g_J}$ and it is $g=g_1+...+g_J$) such that
\begin{enumerate}
 \item $f(\partial\Sigma_{g})=(-L_1)\times\{0\}\sqcup L_2\times\{1\}$.
 \item $f(\Sigma_{g}\setminus\partial\Sigma_{g})\subset S^3\times(0,1)$.
 \item Every connected component of $\Sigma_{g}$ has boundary in both $L_1$ and $L_2$.
\end{enumerate}
In all the figures in this section cobordisms are drawn
as standard surfaces in $S^3$, but they can be knotted in $S^3\times I$.

Some of the induced maps that appear in this subsection come from the work of Sucharit Sarkar 
in \cite{Sarkar}; though the grading shifts are different, because Sarkar used a different definition of the Alexander grading,
ignoring the number of component of the link. 

It is a standard result in Morse theory that a link cobordism can be 
decomposed into five standard cobordisms. We find maps in homology for each case. From now on, given a link
$L_i$, we denote with $D_i$ one of its grid diagrams.
\begin{figure}[ht]
        \centering
        \def\svgwidth{5.5cm}
        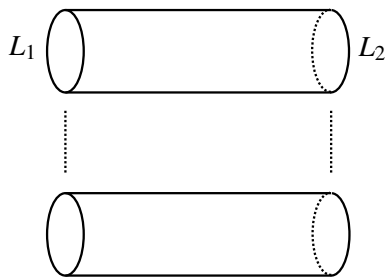   
        \caption{Identity cobordism}
        \label{Identity}
\end{figure}
\begin{enumerate}[i)]
 \item \emph{Identity cobordism}. This cobordism, with no critical points (Figure \ref{Identity}),
       represents a sequence of Reidemeister moves;
       in other words $L_1$ and $L_2$ are smoothly isotopic. 
       At the end of Section \ref{section:homology} we remarked that filtered homology is a link invariant; more precisely
       what we have is a filtered quasi-isomorphism 
       between $\widehat{GC}(D_1)$ and $\widehat{GC}(D_2)$. This, as we know, induces a filtered isomorphism in
       homology. 
       \begin{figure}[ht]      
  \begin{center}
   \begin{tikzpicture}[scale=0.625]
    \draw[help lines] (0,0) grid (2,2);
    \draw[help lines] (-2,0)--(0,0); \draw[help lines] (2,0)--(4,0);
    \draw[help lines] (-2,1)--(0,1); \draw[help lines] (0,1)--(4,1);
    \draw[help lines] (-2,2)--(0,2); \draw[help lines] (0,2)--(4,2);
    \draw[help lines] (0,-2)--(0,0); \draw[help lines] (0,0)--(0,4);
    \draw[help lines] (1,-2)--(1,0); \draw[help lines] (1,0)--(1,4);
    \draw[help lines] (2,-2)--(2,0); \draw[help lines] (2,0)--(2,4);
    \node at (1.5,0.5) {X}; \node at (0.5,1.5) {X};
    \draw (-1.5,1.5) circle [radius=0.3]; \draw (3.5,0.5) circle [radius=0.3];
    \draw (1.5,3.5) circle [radius=0.3];  \draw (0.5,-1.5) circle [radius=0.3];   
    \draw [thick][<-] (0.5,-1.25) -- (0.5,1.25); \draw [thick][<-] (0.25,1.5) -- (-1.25,1.5);
    \draw [thick][<-] (1.5,3.25) -- (1.5,0.75); \draw [thick][<-] (1.75,0.5) -- (3.25,0.5);  
    \draw [thin][->] (5,1) -- (6,1);
    \draw[help lines] (7,0)--(9,0);     \draw[help lines] (9,0)--(13,0);
    \draw[help lines] (7,1)--(9,1);     \draw[help lines] (9,1)--(13,1);
    \draw[help lines] (7,2)--(9,2);     \draw[help lines] (9,2)--(13,2);
    \draw[help lines] (9,-2)--(9,0);    \draw[help lines] (9,0)--(9,4);
    \draw[help lines] (10,-2)--(10,0);  \draw[help lines] (10,0)--(10,4);
    \draw[help lines] (11,-2)--(11,0);  \draw[help lines] (11,0)--(11,4);
    \node at (9.5,0.5) {X}; \node at (10.5,1.5) {X};
    \draw (7.5,1.5) circle [radius=0.3]; \draw (12.5,0.5) circle [radius=0.3];
    \draw (10.5,3.5) circle [radius=0.3];  \draw (9.5,-1.5) circle [radius=0.3];   
    \draw [thick][<-] (9.5,-1.25) -- (9.5,0.25); \draw [thick][<-] (10.25,1.5) -- (7.75,1.5);
    \draw [thick][<-] (10.5,3.25) -- (10.5,1.75); \draw [thick][<-] (9.75,0.5) -- (12.25,0.5);
   \end{tikzpicture}
  \end{center}
  \caption{Band move in a grid diagram}
  \label{Band}
 \end{figure}
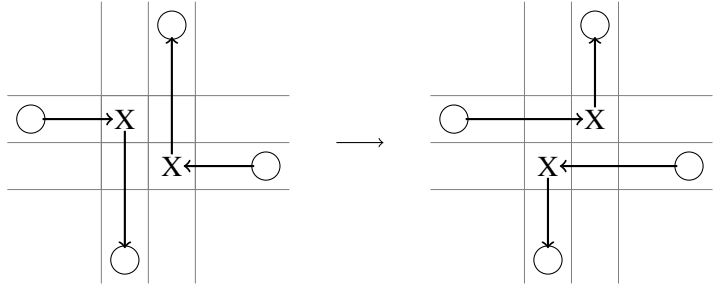
 \item \emph{Split cobordism}. This cobordism (right in Figure \ref{Saddle}) 
       represents a band move when $L_2$ has one more component than $L_1$. Take $D_1$ with a $2\times 2$
       square with two $X$-markings, one at the top-left and one at the bottom-right; then we claim that $D_2$ is obtained
       from $D_1$ by deleting this two $X$-markings and putting two new ones: at the top-right and the bottom-left, as 
       shown in Figure \ref{Band}. 
       \begin{figure}[ht]
        \centering
        \subfigure{
        \def\svgwidth{6cm}
        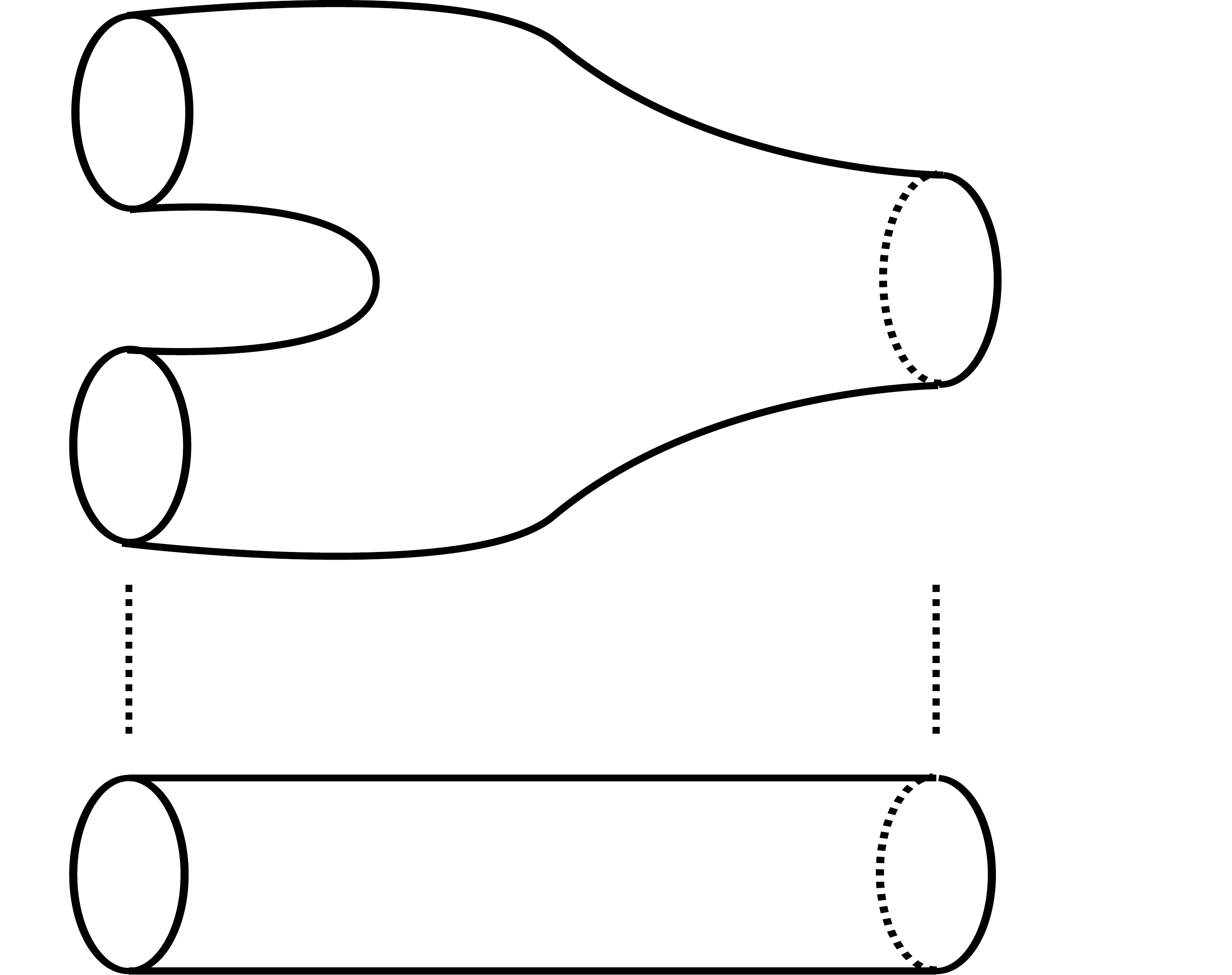}
        \hspace{0cm}
        \subfigure{
        \def\svgwidth{6cm}
        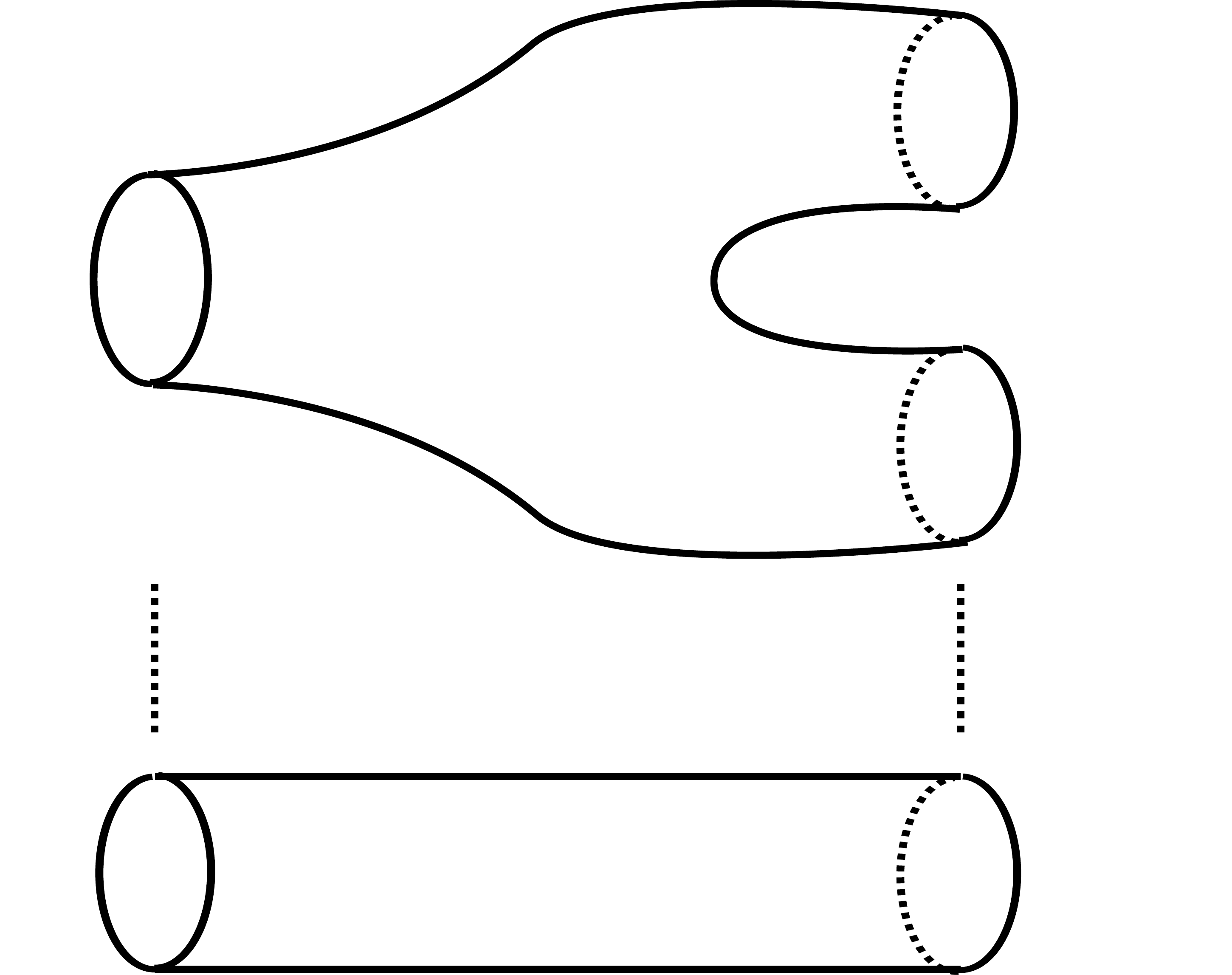}
        \caption{Merge and split cobordisms}
        \label{Saddle}
       \end{figure}
       In order to construct the complex $\widehat{GC}(D_2)$ we need to create one more special
       $O$-marking on the new component of $L_2$. To avoid this problem we first consider the identity map 
       in the filtered $\widetilde{GC}$ theory
       $$\text{Id}:\widetilde{GC}(D_1)\longrightarrow\widetilde{GC}(D_2)\:,$$
       which clearly is a chain map since now every $O$-marking is special; moreover it induces 
       an isomorphism in homology, that preserves the Maslov grading, and a direct computation gives that it is 
       filtered of degree 1.  
       
       Now we use Equation \eqref{tilde} and we get an isomorphism 
       $$\Phi_{\text{Split}}:\widehat{\mathcal{GH}}(L_1)\otimes W\longrightarrow\widehat{\mathcal{GH}}(L_2)$$ 
       that is a degree 1 filtered map which still preserves the Maslov grading. The $W$ factor appears because in 
       Equation \eqref{tilde} we take into account the size of $D_i$ and the number of components of $L_i$; while the 
       first quantity is the same for both diagrams the link $L_2$ has one more component than $L_1$. 
 \item \emph{Merge cobordism}. This cobordism (left in Figure \ref{Saddle})
       represents a band move when $L_1$ has one more component than $L_2$.
       We have an isomorphism
       $$\Phi_{\text{Merge}}:\widehat{\mathcal{GH}}(L_1)\longrightarrow\widehat{\mathcal{GH}}(L_2)\otimes W\:.$$
       The map is obtained in the same way as $\Phi_{\text{Split}}$ in the previous case, 
       but with the difference that now it is filtered of degree 0.
\end{enumerate}
\begin{figure}[ht]
        \centering
        \def\svgwidth{9cm}
        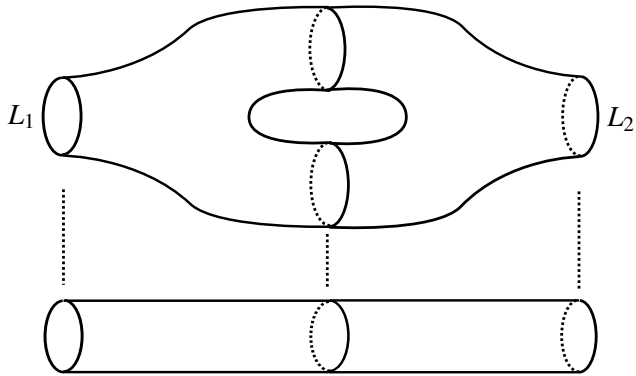   
        \caption{Torus cobordism}
        \label{Torus}
\end{figure}
Sometimes we are more interested in when a split and a merge cobordism appear together, the second just after the first,
in the shape of what we call a torus cobordism (Figure \ref{Torus}). We have the following proposition.
\begin{figure}[ht]
        \centering
        \def\svgwidth{6cm}
        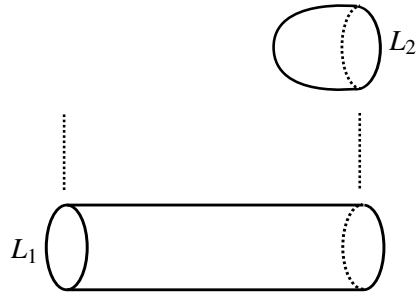   
        \caption{Birth cobordism}
        \label{Birth}
\end{figure}
\begin{prop}
 \label{prop:torus}
 Let $\Sigma$ be a torus cobordism between two links $L_1$ and $L_2$. Then $\Sigma$ induces a Maslov grading preserving
 isomorphism between $\widehat{\mathcal{GH}}(L_1)$ and $\widehat{\mathcal{GH}}(L_2)$, which is filtered of degree 1 .
\end{prop}
\begin{proof}
 We can choose $D_1$ in a way that the split and the merge band moves can be performed on two disjoint bands.
 Then we apply twice the move shown in Figure \ref{Band} and we take as map the identity. In this case the identity is a 
 chain map 
 because $L_2$ has the same number of components of $L_1$; this means that the special $O$-markings in $D_1$ and $D_2$ are
 the same and then the two differentials coincide.
 In this way we obtain an isomorphism in homology with Maslov grading shift and 
 filtered degree equal to the sum of the ones in ii) and iii).
\end{proof}
\begin{enumerate}[iv)]
 \item \emph{Birth cobordism}. A cobordism (Figure
       \ref{Birth}) representing a birth move.  
\end{enumerate}       
       Since our cobordisms have boundary in both $L_1$ and $L_2$, we can always assume that a birth move is followed 
       (possibly after some Reidemeister moves) by a merge move. Thus it is enough to define a map for the composition of these
       three cobordisms and this is what we do in the following proposition.
\begin{prop}
 \label{prop:birth}
 Let $\Sigma$ be a cobordism between two links $L_1$ and $L_2$ like the one 
 in Figure \ref{Birth1}. Then $\Sigma$ induces an 
 isomorphism $\Phi_{\text{Birth}}$ 
 between $\widehat{\mathcal{GH}}(L_1)$ and $\widehat{\mathcal{GH}}(L_2)$ that preserves the Maslov grading and
 it is filtered of degree 0.
\end{prop}
\begin{figure}[ht]
        \centering
        \def\svgwidth{9cm}
        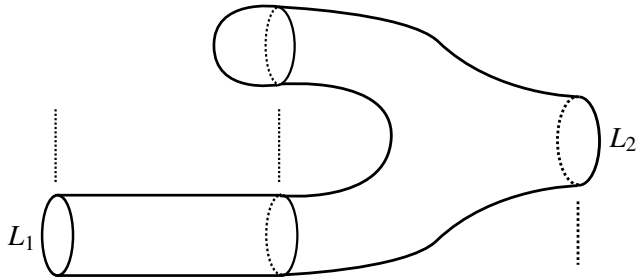   
        \caption{A more useful birth cobordism}
        \label{Birth1}
\end{figure}
\begin{proof}
The first step is to construct a map $s_1$ associated to the grid move shown in Figure \ref{birth_grid}; note that we add a 
normal $O$-marking, since later we merge the new unknot component with an already exisiting one.
Let us denote with $D_1'$ the stabilized diagram and with $c=\alpha\cap\beta$ the point in the picture; then we have the 
inclusion $i:S(D_1)\rightarrow S(D_1')$ that sends a grid state $x$ in $D_1$ to the grid state in $D_1'$ 
constructed
from $x$ by adding the point $c$. Then $s_1:\widehat{GC}(D_1)\rightarrow\widehat{GC}(D_1')$ is defined by the following 
formula:
$$s_1(x)=\sum_{y\in S(D_1')}\sum_{\substack{H\in\mathcal{SL}(i(x),y,c) \\ 
H\cap s\OO=\emptyset}}V_1^{n_1(H)}\cdot...\cdot V_m^{n_m(H)}y\:\:\:\:\:\text{for any }x\in S(D_1)$$
\begin{figure}[ht]      
  \begin{center}
   \begin{tikzpicture}[scale=0.625]
    \draw[help lines] (0,0.5)--(3,0.5); \draw[help lines] (0,0.5)--(0,3.5);
    \draw[help lines] (0,3.5)--(3,3.5); \draw[help lines] (3,0.5)--(3,3.5);
    \node at (1.5,2) {$D_1$}; 
    \draw [thin][->] (4,2) -- (5,2);
    \draw[help lines] (6,0)--(10,0); \draw[help lines] (6,0)--(6,4);
    \draw[help lines] (6,3)--(10,3); \draw[help lines] (9,0)--(9,4);
    \draw[help lines] (6,4)--(10,4); \draw[help lines] (10,0)--(10,4);
    \node at (7.5,1.5) {$D_1$}; 
    \node at (9.5,3.5) {X}; \draw (9.5,3.5) circle [radius=0.3]; 
    \draw[fill] (9,3) circle [radius=0.1];
    \node[left] at (6,3) {$\alpha$}; \node[below] at (9,0) {$\beta$};
    \node[above left] at (9,3) {$c$};
   \end{tikzpicture}
  \end{center}
  \caption{Birth move in a grid diagram}
  \label{birth_grid}
\end{figure}
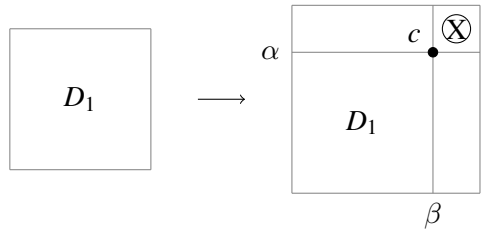
where $\mathcal{SL}(x,z,p)$ is the set of all the snail-like domains (the exact definition can be found in \cite{Book} 
Chapter 13) centered at $p$ joining $x$ to $z$, illustrated in
Figure \ref{snail};
$n_i(H)$ is the number of times $H$ passes through $O_i$ and $m$ is the number of the normal $O$-markings of $D_1$.
In \cite{Manolescu1} is proved that $s_1$ is a filtered quasi-isomorphism which induces a filtered isomorphism between
$\widehat{\mathcal{GH}}(L_1)\rightarrow\widehat{\mathcal{GH}}^{\OO}\left(L_1\sqcup\bigcirc\right)$; where the special
$O$-markings on $L_1\sqcup\bigcirc$ coincide with the ones on $L_1$ (the new unknotted component has a normal $O$-marking).
Moreover, in \cite{Sarkar} Sarkar showed that the map $s_1$ is filtered of degree 0. 

At this point we compose $s_1$ with the map $s_2$ given by the Reidemeister moves, which is a filtered
quasi-isomorphism for Theorem \ref{teo:new}, and finally with $s_3$,
the identity associated to the band move of Figure \ref{Band}. 
\begin{figure}[ht]
        \centering
        \def\svgwidth{16cm}
        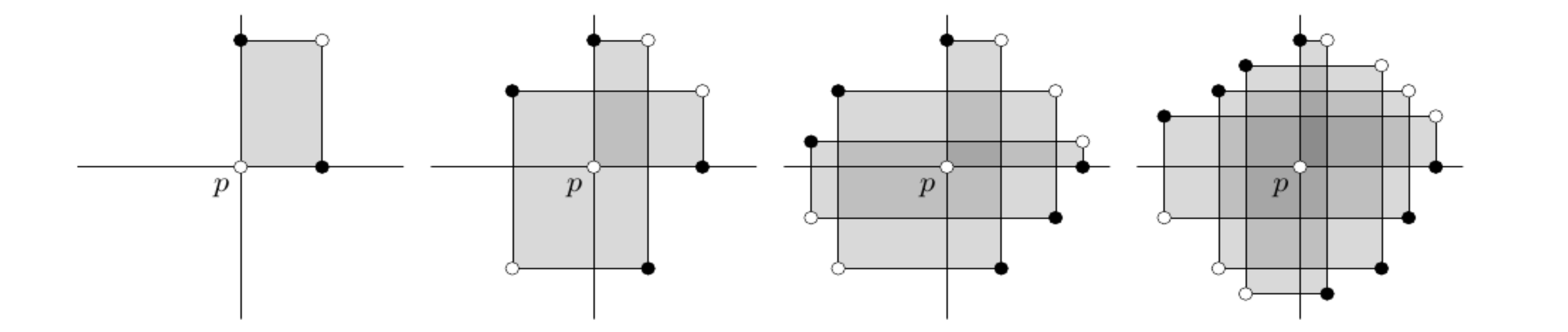   
        \caption{Some of the snail-like domains $\mathcal{SL}(x,z,p)$: the coordinates of $x$ and $z$ are represented
                by the white and black circles}
        \label{snail}
\end{figure}
The map $s_3$ induces an isomorphism 
$\widehat{\mathcal{GH}}^{\OO}\left(L_1\sqcup\bigcirc\right)\rightarrow\widehat{\mathcal{GH}}(L_2)$ that clearly preserves the 
Maslov grading and again we easy compute that it is filtered of degree 0.

Hence, the composition of these three maps that we defined induces the isomorphism in the claim.
\end{proof}
\begin{enumerate}[v)]
 \item \emph{Death cobordism}. This cobordism (Figure \ref{Death}) 
 represents a death move. Since this move can also be seen as a birth move between $L_2^*$ and
 $L_1^*$, we take the dual map of 
 $$\Phi_{\text{Birth}}:\widehat{\mathcal{GH}}(L_2^*)\longrightarrow\widehat{\mathcal{GH}}(L_1^*)$$
 which exists from Proposition \ref{prop:birth}; then $\Phi_{\text{Birth}}^*=\Phi_{\text{Death}}$, 
 by Proposition \ref{prop:mirror}, is a map
 between $\widehat{\mathcal{GH}}(L_1)$ and $\widehat{\mathcal{GH}}(L_2)$. Furthermore, it is still an isomorphism that is 
 filtered of degree 0 and preserves the Maslov grading.
 \begin{figure}[ht]
        \centering
        \def\svgwidth{9cm}
        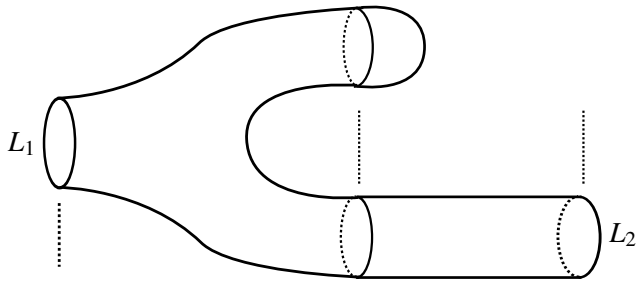   
        \caption{Death cobordism}
        \label{Death}
 \end{figure}
\end{enumerate}
The results for birth and death cobordisms in this section immediately give the following corollary.
\begin{cor}
 \label{cor:new}
 Suppose there is a birth or a death cobordism as in Figures \ref{Birth1} and \ref{Death} between two links $L_1$ and $L_2$. 
 Then we have that $\widehat{\mathcal{GH}}(L_1)\cong\widehat{\mathcal{GH}}(L_2)$.
\end{cor}

\subsection{Strong concordance invariance}
We want to prove Theorem \ref{teo:one}.  
We remark that a strong cobordism is a cobordism $\Sigma$, between two links with the same number of components, such that 
every connected component of 
$\Sigma$ is a knot cobordism between a component of the first link and one of the second link.
Moreover, if the connected components of $\Sigma$ are all annuli then we call $\Sigma$ a strong concordance.
We start by observing that Proposition \ref{prop:torus} leads to the following corollary.
\begin{cor}
 \label{cor:torus}
 Suppose there is a strong cobordism $\Sigma$ between $L_1$ and $L_2$ such that 
 $\Sigma$ is the composition of $g(\Sigma)$ torus
 cobordisms, not necessarily all of them belonging to the same component of $\Sigma$.
 Then $\Sigma$ induces an isomorphism between $\widehat{\mathcal{GH}}(L_1)$ 
 and $\widehat{\mathcal{GH}}(L_2)$, which is filtered of degree $g(\Sigma)$ and preserves the Maslov grading.
\end{cor}
Now, if we have an isomorphism $F:\widehat{\mathcal{GH}}(L_1)\rightarrow\widehat{\mathcal{GH}}(L_2)$ that preserves the
Maslov grading and it is filtered of degree $t$, which means that 
there are inclusions 
$F\left(\mathcal F^s\widehat{\mathcal{GH}}_d(L_1)\right)\subset\mathcal F^{s+t}\widehat{\mathcal{GH}}_d(L_2)$ for every 
$d,s\in\Z$, then $\tau(L_2)\leq\tau(L_1)+t$.
Hence we can prove the following theorem that immediately implies the invariance statement.
\begin{teo}
 \label{teo:strong}
 Suppose that $\Sigma$ is a strong cobordism between two links $L_1$ and $L_2$. Then
 $$\big|\tau(L_1)-\tau(L_2)\big|\leq g(\Sigma)\:.$$
 Furthermore, if $L_1$ and $L_2$ are strongly concordant
 then $\widehat{\mathcal{GH}}(L_1)\cong\widehat{\mathcal{GH}}(L_2)$.
\end{teo}
\begin{proof}
 By Proposition B.5.1 in \cite{Book} we can suppose that, 
 in $\Sigma$, 0-handles come before 1-handles while 2-handles come later; moreover we can say that 
 $\Sigma$ is the composition of birth, torus and death cobordisms (and obviously some identity cobordisms).
 Each of these induces an isomorphism in homology that also respects the Maslov grading.
 
 For the first part, we only need to check what is the filtered degree of the isomorphism between 
 $\widehat{\mathcal{GH}}(L_1)$ and $\widehat{\mathcal{GH}}(L_2)$, obtained by the composition of all the induced maps on
 each piece of $\Sigma$. Birth, death and identity are filtered of degree 0, while, from Corollary \ref{cor:torus},
 the torus cobordisms are filtered of degree $g(\Sigma)$. Then we obtain
 $$\tau(L_2)\leq\tau(L_1)+g(\Sigma)\:.$$ 
 For the other inequality we consider the same cobordism, but this time from $L_2$ to $L_1$.
 
 Now, for the second part, we observe that now there are no torus cobordisms
 and then the claim follows from Corollary \ref{cor:new}.
\end{proof}

\subsection{A lower bound for the slice genus}
Suppose $\Sigma$ is a cobordism (not necessarily strong) between two links $L_1$ and $L_2$. Denote with 
$\Sigma_1,...,\Sigma_J$ the connected components of $\Sigma$. For $i=1,2$ we define the integers $l_i^k(\Sigma)$ as
the number of components of $L_i$ that belong to $\Sigma_k$ minus 1; in particular $l_i^k(\Sigma)\geq 0$ for any $k,i$.
Finally, we say that $l_i(\Sigma)=\displaystyle\sum_{k=1}^{J}l_i^k(\Sigma)=n_i-J$, where $n_i$ is the number of component 
of $L_i$. For example, if $\Sigma$ is the cobordism in Figure \ref{Big} then $J=2$, ordering $\Sigma_1$ and
$\Sigma_2$ from above to bottom, we have that $l_1(\Sigma)=3$ and
$l_2(\Sigma)=4$; while $l_1^1(\Sigma)=2$, $l_1^2(\Sigma)=1$, $l_2^1(\Sigma)=3$ and $l_2^2(\Sigma)=1$.
We have the following lemma.
\begin{lemma}
 \label{lemma:normal}
 If $\Sigma$ has no 0,2-handles then, up to rearranging 1-handles, we can suppose that $\Sigma$ is like in Figure \ref{Big}:
 there are $l_1(\Sigma)$ merge cobordisms between $(0,t_1)$, $l_2(\Sigma)$ split cobordisms between $(t_2,1)$ and 
 $g(\Sigma)$ torus cobordisms between $(t_1,t_2)$. We have no other 1-handles except for the ones we considered
before.
\end{lemma}
\begin{figure}[ht]
        \centering
        \def\svgwidth{14cm}
        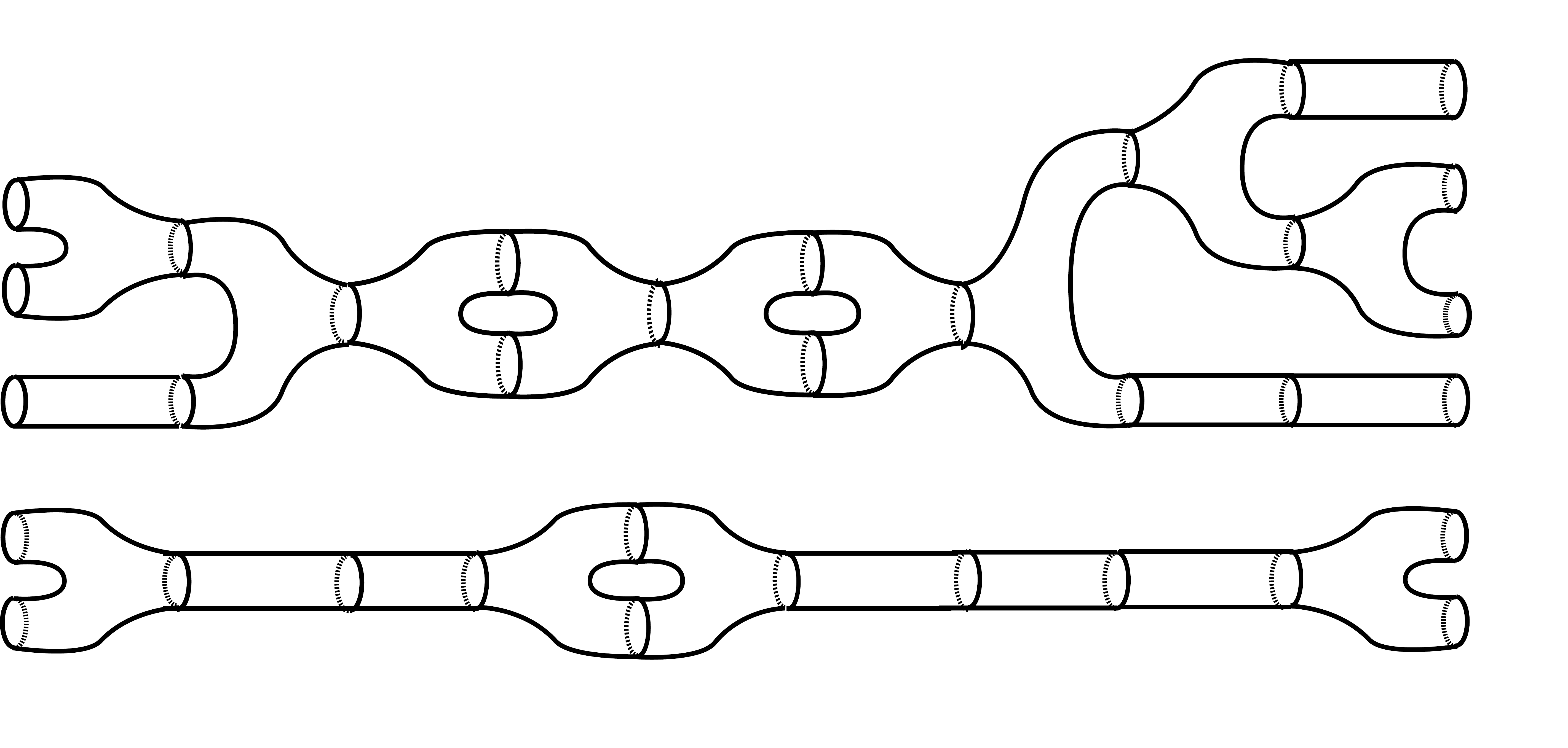   
        \caption{}
        \label{Big}
\end{figure}
\begin{proof}
 We consider a connected component $\Sigma_k$, which is a cobordism between $L_1^k$ and $L_2^k$, and we fix a Morse 
 function $f:S^3\times[0,1]\rightarrow[0,1]$. After some Reidemeister
 moves, by Proposition B.5.1
 we can assume that all the band moves are performed on disjoint bands, in particular we can apply them in every
 possible order.
 
 Since $\Sigma_k$ has boundary in both $L_1^k$ and $L_2^k$ by the definition of cobordism given at beginning
 of Section \ref{section:cobordism}, if we take an ordering for the components of $L_1^k$, we can find
 a merge cobordism joining the first and the second component of $L_1^k$ at some point $\overline t$ in $(0,1)$; we assume 
 that the associated band move is the first we apply on $L_1^k$. Now we just do the same thing on the other components,
 but taking the new component instead of the first two. In this way 
 we have that there is a $t_1\in(0,1)$ such that $\Sigma_k\cap f^{-1}[0,t_1]$ is composed
 by $l_1^k(\Sigma)$ merge cobordisms and $\Sigma_k\cap f^{-1}(t_1)$ is a knot.
 
 In the same way we find that, for a certain $t_2\in(0,1)$, the cobordism $\Sigma_k\cap f^{-1}[t_2,1]$ is composed by 
 $l_2^k(\Sigma)$ split cobordisms and $\Sigma_k\cap f^{-1}(t_2)$ is a knot.
 
 At this point $\Sigma_k\cap f^{-1}[t_1,t_2]$ is a knot cobordism of genus $g(\Sigma_k)$ and 
 from Lemma B.5.3 in \cite{Book} (see also \cite{Kawauchi}) we can rearrange
 the saddles to obtain a composition of $g(\Sigma_k)$ torus cobordisms like in Figure \ref{Big}.
 
 To see that there are no other 1-handles left it is enough to compute the Euler characteristic of $\Sigma_k$:
 $$2-2g(\Sigma_k)-l_1^k(\Sigma)-1-l_2^k(\Sigma)-1=\chi(\Sigma_k)=-\#|\text{1-handles}|\:.$$
 This means that the number of 1-handles in $\Sigma_k$ is precisely $2g(\Sigma_k)+l_1^k(\Sigma)+l_2^k(\Sigma)$. 
\end{proof}
From Figure \ref{Big} we realize that merge and split cobordisms can appear alone in $\Sigma$ and not always in pair like
in strong cobordisms. In case ii) and iii) of Subsection \ref{subsection:shift} we see that they do not induce 
isomorphisms in homology, but we find maps $\Phi_{\text{Split}}$ and $\Phi_{\text{Merge}}$ 
that are indeed isomorphisms if restricted to 
$\widehat{\mathcal{GH}}_0(L_1)\rightarrow\widehat{\mathcal{GH}}_0(L_2)$; 
moreover, the filtered degree is 1 for split cobordisms and 0 for merge cobordisms. Since we are looking for informations
on $\tau$, this is enough for our goal and then we can prove the following inequality.
\begin{prop}
 \label{prop:bound}
 Suppose $\Sigma$ is a cobordism between two links $L_1$ and $L_2$. Then 
 $$\big|\tau(L_1)-\tau(L_2)\big|\leq g(\Sigma)+\text{max }\{l_1(\Sigma),l_2(\Sigma)\}\:.$$
\end{prop}
\begin{proof}
 By Theorem \ref{teo:strong},
 we can suppose that there are no 0 and no 2-handles in $\Sigma$.
 We can also assume
 that $\Sigma$ is like in Lemma \ref{lemma:normal}.
 
 All of these
 cobordisms induce isomorphisms of the homology in Maslov grading 0. The number
 of torus cobordisms is $g(\Sigma)$ while the number of split cobordisms 
 (that are not part of torus cobordisms)
 is $l_2(\Sigma)$. This means that
 $$\tau(L_2)\leq\tau(L_1)+g(\Sigma)+l_2(\Sigma)\:.$$
 Now we do the same, but considering the cobordism going from $L_2$ to $L_1$, as we did in the proof of Theorem 
 \ref{teo:strong}. We obtain that 
 $$\tau(L_1)\leq\tau(L_2)+g(\Sigma)+l_1(\Sigma)\:.$$
 Putting the two inequalities together proves the relation in the statement of the theorem.
\end{proof}
If $L$ is an $n$-component link, from Proposition \ref{prop:bound} we have immediately Equation \eqref{bound}:
\begin{equation*}
 |\tau(L)|+1-n\leq g_4(L)
\end{equation*}
which, as we already said,
is a lower bound for the slice genus of our link. Indeed, we can say more by using Equation \eqref{tau_bar}
and observing that $g_4(L^*)=g_4(L)$:
\begin{equation*}
 \text{max }\big\{|\tau(L)|,|\tau^*(L)|\big\}+1-n\leq g_4(L)\:.
\end{equation*}

\section{The \texorpdfstring{$\mathcal{GH}^-$}{Lg} version of filtered grid homology}
\label{section:minus}
\subsection{A different point of view}
The collapsed filtered complex $cGC^-(D)$ for a grid diagram $D$ is the
free $\F[V_1,...,V_{\text{grd}(D)-n},V]$-module generated by the set of grid states $S(D)$. This ring has one more 
variable $V$, compared to the ring we considered for $\widehat{GC}(D)$, associated to the special $O$-markings.
The differential $\partial^-$ is defined
as following:
$$\partial^-x=\sum_{y\in S(D)}\sum_{r\in\text{Rect}^{\circ}(x,y)}V_1^{O_1(r)}\cdot...\cdot V_m
^{O_m(r)}\cdot V^{O(r)}y\:\:\:\:\text{for any }x\in S(D)$$
where $m=\text{grd}(D)-n$ and $O(r)$ is the number of special $O$-markings in $r$.

It is clear from the definition that 
$$\left(\widehat{GC}(D),\widehat{\partial}\right)=\frac{\big(cGC^-(D),\partial^-\big)}{V=0}\:.$$
The collapsed filtered
unblocked homology $c\mathcal{GH}^-(L)$ is the homology of our new complex and it is a link invariant; but 
now each level $\mathcal F^s c\mathcal{GH}^-(L)$ has also a structure of an $\F[U]$-module given by $U[p]=[V_ip]=[Vp]$ for 
every $i=1,...,m$
and $[p]\in c\mathcal{GH}^-(L)$.
The groups $\mathcal F^s c\mathcal{GH}^-_d(L)$ are still finite dimensional over $\F$ and so we can define the 
function $N$
as $$N_L(d,s)=\text{dim}_{\F}\frac{\mathcal F^s c\mathcal{GH}^-_d(L)}{\mathcal F^{s-1}c\mathcal{GH}^-_d(L)}\:.$$
We expect the function $N$ to be a strong concordance invariant, possibly better than $T$.

We can compute the function $N$ of the unknot:
$$N_{\bigcirc}(d,s)=\left\{\begin{aligned}
                                 &1\:\:\:\:\:\text{if }(d,s)=(2t,t),\:\:\:t\leq 0 \\
                                 &0\hspace{1.5cm}\text{otherwise}
                                \end{aligned}\right.$$
and we know that
$$c\mathcal{GH}^-(L)\cong_{\F[U]}c\mathcal{GH}^-(\bigcirc_n)\cong_{\F[U]}\F[U]^{2^{n-1}}$$ as an 
$\F[U]$-module, where $n$ is the number of component of $L$.

Since $\text{dim}_{\F}c\mathcal{GH}^-_0(L)$ is still equal to 1, we can define an invariant $\nu$ exactly like we did in 
Subsection
\ref{subsection:definitions} for $\tau$. 
A version of the $\nu$-invariant has been introduced first by Jacob Rasmussen in \cite{Rasmussen}
and he proved that it is a concordance invariant for knots.
In \cite{Hom} Hom and Wu found knots whose $\nu$-invariant gives better lower bound for the slice genus than $\tau$.

Since $H_{*,*}\left(\text{gr}\left(cGC^-(D)\right)\right)$ is isomorphic to the homology $cGH^-(L)$ of \cite{Book}
and $N_L(0,\nu(L))=1$ for every diagram $D$ of $L$,
we have that $cGH^-_{0,\nu(L)}(L)$ is non-trivial. Hence, if the homology group $cGH^-_0(L)$ is non-zero only for one 
Alexander grading $s$, we can argue that $\nu(L)=s$. This method can be used to compute the $\nu$-invariant of some
links.

We can also define the (uncollapsed) filtered unblocked homology as the homology of the complex
$\big(GC^-(D),\partial^-\big)$, 
where $GC^-(D)$ is the free $\F[V_1,...,V_{\text{grd }(D)}]$-module over the grid states of $D$;
there are no special $O$-markings this time.

We see immediately that for every $n$-component link $L$ is
$$\mathcal{GH}^-(L)\cong\F[U_1,...,U_n]$$
with generator in Maslov grading 0, but the filtration will of course depend on $L$.

\subsection{Proof of Theorem \ref{teo:set}}
We use the complex $cGC^-$ to prove that, for every $n$-component link $L$, an integer $s$ gives $T_{d,s}(L)\neq0$ for some 
$d$ if and only if $s$ belongs to the $\tau$-set of $L$. 
To do this, given $\mathcal C$ a freely and finitely generated $\F[U]$-complex, we define two set of integers:
$\tau(\mathcal C)$ and $t(\mathcal C)$. First we call $\mathcal B_{\tau}(\mathcal C)$ a homogeneous,
free generating set of the torsion-free quotient
of $H_{*,*}\left(\text{gr}(\mathcal C)\right)$ as an $\F[U]$-module;
then $\tau(\mathcal C)$ is the set of $s\in\Z$ such that there is a $[p]\in\mathcal B_{\tau}(\mathcal C)$
with bigrading $(d,-s)$ for some $d\in\Z$.
Similarly,
$t(\mathcal C)$ is the set of the integers
$s$ such that the inclusion $$i_s:\mathcal F^{s-1}H_*\left(\frac{\mathcal{C}}{U=0}\right)
\lhook\joinrel\relbar\joinrel\rightarrow
\mathcal F^sH_*\left(\frac{\mathcal{C}}{U=0}\right)$$ is not surjective. Note that the set 
$\mathcal B_{\tau}(\mathcal C)$ is not 
unique, but $\tau(\mathcal C)$ and $t(\mathcal C)$ are well-defined.

We say that $s\in\tau(\mathcal C)$ has multiplicity $k$ if there are $k$ distinct elements in $\mathcal B_{\tau}(\mathcal C)$
with bigrading $(*,-s)$, while a number $u\in t(\mathcal C)$ has multiplicity $k$ if $\text{Coker }i_u$ has dimension $k$ as
an $\F$-vector space.

Clearly, if $D$ is a grid diagram of $L$,
$t\big(cGC^-(D)\big)$ is the set of the values of $s$
where the function $T_L$ is supported; moreover, we already remarked that 
$H_{*,*}\left(\text{gr}\left(cGC^-(D)\right)\right)$ is isomorphic to $cGH^-(L)$ and then
$\tau\big(cGC^-(D)\big)$ is the $\tau$-set of $L$. Hence our goal is to prove
that $t\big(cGC^-(D)\big)=\tau\big(cGC^-(D)\big)$, generalizing Proposition 14.1.2 in \cite{Book}.

Consider the complex $$\mathcal C=\frac{cGC^-(D)}{V_1=...=V_m=U}\:,$$ where $m=\text{grd}(D)-n$ and $U$ is the variable 
associated to the special $O$-markings. Then we define $\overline{\mathcal C}$ as the complex 
$\mathcal C\llbracket 1-n-m,-m\rrbracket$.

We introduce a new complex $\mathcal C'=\mathcal C\otimes_{\F[U]}\F[U,U^{-1}]$ and we define a $\Z\oplus\Z$-filtration on
$\mathcal C'$ in the following way: $\mathcal F^{x,*}\mathcal{C'}=U^{-x}\mathcal C$ for every $x\in\Z$, 
$\mathcal F^{*,y}\mathcal C'=\mathcal F^y\mathcal C'=\big\{p\in\mathcal C'\:|\:A(p)\leq y\big\}$ for every $y\in\Z$ and 
$\mathcal F^{x,y}\mathcal C'=\mathcal F^{x,*}\mathcal C'\cap\mathcal F^{*,y}\mathcal C'=U^{-x}\mathcal F^{y-x}\mathcal C$
for every $x,y\in\Z$.
The first step is to prove that $\tau(\mathcal C)=t(\overline{\mathcal C})$.

We have that $$\frac{\mathcal C}{U=0}\cong\widetilde{GC}(D)
\cong\frac{\mathcal F^{0,*}\mathcal C'}{\mathcal F^{-1,*}\mathcal C'}\:;$$ moreover,
for every integer $s$ we claim that 
$$\mathcal F^s\widetilde{GC}(D)\cong\frac{\mathcal F^{0,s}\mathcal C'}{\mathcal F^{-1,s}\mathcal C'}\:.$$
Since, from Lemma 14.1.9 in \cite{Book}, it is $\mathcal F^{x,y}\mathcal C'\cong\mathcal F^{y,x}\overline{\mathcal C'}$
for every $x,y\in\Z$, then we have that 
\begin{equation}
 \label{new}
 \mathcal F^s\frac{\overline{\mathcal C}}{U=0}\cong\mathcal F^s
 \widetilde{GC}(D)\llbracket1-n-m,-m\rrbracket
 \cong\frac{\mathcal F^{s,0}\mathcal C'}{\mathcal F^{s,-1}\mathcal C'}\:.
\end{equation}
Using this identification we obtain that $t(\overline{\mathcal C})$ coincides with set of $s\in\Z$ such that the map
\begin{equation}
 \label{tau_inc}
 H_*\left(\frac{\mathcal F^{s-1,0}\mathcal C'}{\mathcal F^{s-1,-1}\mathcal C'}\right)
 \lhook\joinrel\relbar\joinrel\rightarrow H_*\left(\frac{\mathcal F^{s,0}\mathcal C'}{\mathcal F^{s,-1}\mathcal C'}\right)
\end{equation}
is not surjective.

Now we consider the complex $\text{gr}(\mathcal C)$, which is equal to $\displaystyle\bigoplus_{t\in\Z}
\dfrac{\mathcal F^{0,t}\mathcal C'}{\mathcal F^{0,t-1}\mathcal C'}$. We have that the map
$$U^t:\frac{\mathcal F^{0,t}\mathcal C'}{\mathcal F^{0,t-1}\mathcal C'}\longrightarrow\frac{\mathcal F^{-t,0}\mathcal C'}
{\mathcal F^{-t,-1}\mathcal C'}$$ is an isomorphism and $\dfrac{\mathcal F^{-t,0}\mathcal C'}
{\mathcal F^{-t,-1}\mathcal C'}$ is a subspace of $\dfrac{\mathcal F^{*,0}\mathcal C'}
{\mathcal F^{*,-1}\mathcal C'}$ for every integer $t$. From Lemma 14.1.12 in \cite{Book} the latter filtered 
complex is isomorphic
to $\dfrac{\text{gr}(\mathcal C)}{U=1}$, but it is also isomorphic to  
$\widetilde{GC}(D)\llbracket1-n-m,-m\rrbracket$ for Equation \eqref{new}. In this way we can define 
a surjective map
$$\Psi:\text{gr}(\mathcal C)\longrightarrow\frac{\text{gr}(\mathcal C)}{U=1}$$ and it is easy to see that 
$\Psi\left(\mathcal B_{\tau}(\mathcal C)\right)$ is still a homogeneous, free generating set of the homology; furthermore,
if $[p]$ is a torsion element in $H_{*,*}(\text{gr}(\mathcal C))$ then $[\Psi(p)]=[0]$. This means that 
$\tau(\mathcal C)$ is the set of $-t\in\Z$ such that the map
\begin{equation}
 \label{t_inc}
 H_*\left(\frac{\mathcal F^{-t-1,0}\mathcal C'}{\mathcal F^{-t-1,-1}\mathcal C'}\right)
 \lhook\joinrel\relbar\joinrel\rightarrow H_*\left(\frac{\mathcal F^{-t,0}\mathcal C'}{\mathcal F^{-t,-1}\mathcal C'}\right)
\end{equation}
is not surjective. 

If we change $-t$ with $s$ in Equation \eqref{t_inc} then we immediately see that it coincides with Equation \eqref{tau_inc} 
and so $\tau(\mathcal C)=t(\overline{\mathcal C})$. Moreover, we can say that an integer in $\tau(\mathcal C)$ has
multiplicity $k$ if and only if it has multiplicity $k$ in $t(\overline{\mathcal C})$. Finally, since the map $U^t$ drops
the Maslov grading by $2t$, we have that if there is a $[p]\in\mathcal B_{\tau}(\mathcal C)$ with bigrading $(d,s)$ then 
there is a generator of $\widetilde{\mathcal{GH}}(D)\llbracket 1-n-m,-m\rrbracket$ with 
grading and minimal level $(d-2s,-s)$.

The second and final step is to show that the previous claim implies $t\big(cGC^-(D)\big)=\tau\big(cGC^-(D)\big)$.
From Lemma 14.1.11 in \cite{Book} we have the filtered
quasi-isomorphisms
$\mathcal C\cong cGC^-(D)\otimes W^{\otimes m}$
and $\overline{\mathcal C}\cong cGC^-(D)\otimes(W^*)^{\otimes m}$, where
$W$ is the two dimensional $\F$-vector space with generators in grading and minimal 
level $(d,s)=(0,0)$ and $(d,s)=(-1,-1)$.
Thus $t\big(cGC^-(D)\big)$ and $\tau\big(cGC^-(D)\big)$ completely determine $\tau(\mathcal C)$ and 
$t(\overline{\mathcal C})$, so this 
means that they coincide and the proof is complete. Obviously, the conclusions about multiplicities and Maslov shifts are 
still true.
 
From \cite{Book} Chapter 8
we know that there is only one element  $[p]\in\mathcal B_{\tau}\left(cGC^-(D)\right)$ with bigrading 
$(-2\tau_1,-\tau_1)$ and only another one $[q]$ with bigrading $(-2\tau_2+1-n,-\tau_2)$. Then the proof of Theorem 
\ref{teo:set} implies that there are two non-zero elements in $\widehat{\mathcal{GH}}(L)$ in grading and minimal level
$(0,\tau_1)$ and $(1-n,\tau_2)$. Since, from the definition of $\tau$ and $\tau^*$, we also know that there are only two 
generators of $\widehat{\mathcal{GH}}(L)$ in Maslov grading $0$ and $1-n$; we have the following corollary.
\begin{cor}
 \label{cor:twice}
 Take a grid diagram $D$ of an $n$-component link $L$ and consider a set $\mathcal B_{\tau}\left(cGC^-(D)\right)$. 
 If $[p],[q]\in\mathcal B_{\tau}\left(cGC^-(D)\right)$ are such that $[p]$ is in bigrading $(-2\tau_1,-\tau_1)$ and 
 $[q]$ is in bigrading $(-2\tau_2+1-n,-\tau_2)$ then $\tau(L)=\tau_1$ and $\tau^*(L)=\tau_2$.
\end{cor}

\section{Applications}
\label{section:application}
\subsection{Computation for some specific links}
In general it is hard to say when a sum of grid states is a generator of the homology, but the following lemma 
provides an example where we have useful information.
\begin{lemma}
 \label{lemma:generator}
 Suppose $L$ is an $n$-component link with grid diagram $D$ and $x\in S(D)$ as in Figure \ref{1}.
 Then $[x]$ is always the generator of $\widehat{\mathcal{GH}}_0(L)$ and
 $c\mathcal{GH}^-_0(L)$. Furthermore, $\tau(L)=\nu(L)=A(x)$.
\end{lemma}
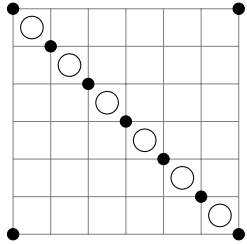
\begin{figure}[ht]
  \begin{center}
   \begin{tikzpicture}[scale=0.5]
    \draw[help lines] (0,0) grid (6,6); 
    \draw[fill] (0,0) circle [radius=0.15]; \draw (0.5,5.5) circle [radius=0.3];
    \draw[fill] (0,6) circle [radius=0.15]; \draw (1.5,4.5) circle [radius=0.3];
    \draw[fill] (1,5) circle [radius=0.15]; \draw (2.5,3.5) circle [radius=0.3];
    \draw[fill] (2,4) circle [radius=0.15]; \draw (3.5,2.5) circle [radius=0.3];
    \draw[fill] (3,3) circle [radius=0.15]; \draw (4.5,1.5) circle [radius=0.3];
    \draw[fill] (4,2) circle [radius=0.15]; \draw (5.5,0.5) circle [radius=0.3];
    \draw[fill] (5,1) circle [radius=0.15];
    \draw[fill] (6,0) circle [radius=0.15];
    \draw[fill] (6,6) circle [radius=0.15];
   \end{tikzpicture}
  \end{center}
  \caption{We denote with $x$ the grid state in the picture}
  \label{1}
 \end{figure}
\begin{proof}
 We show that $M(x)=0$, $\widehat{\partial}x=\partial^-x=0$ and that for every other grid state $y$ of $D$ it is
 $M(y)\leq 0$.
 \begin{enumerate}[i)]
  \item The fact that $M(x)=0$ is trivial.
  \item For every $y\in S(D)$ there are always 2 
        rectangles in $\text{Rect}^{\circ}(x,y)$ and they contain no $O$, so they cancel when we 
        compute the differential.
  \item We prove by induction on $\text{grd}(D)$ that $M(y)\leq 0$ for every $y\in S(D)$.   
        
        If $\text{grd}(D)=1$ then $x$ is the only grid state.
        
        If $\text{grd}(D)=2$ then there are only $x$ and $y$ and it is $M(y)=-1$.
        
        Suppose the claim is true for the diagrams with dimension equal or smaller than $\alpha$ and let 
        $\text{grd}(D)=\alpha+1$.
  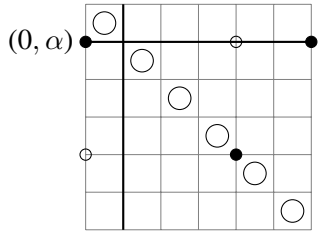
\begin{figure}[ht]      
  \begin{center}
   \begin{tikzpicture}[scale=0.5]
    \draw[help lines] (0,0) grid (6,6); 
    \node[left] at (0,5) {$(0,\alpha)$};
    \draw[thick] (1,0)--(1,6); \draw[thick] (0,5)--(6,5);
    \draw[fill] (0,5) circle [radius=0.15]; \draw (0,2) circle [radius=0.15]; 
    \draw[fill] (6,5) circle [radius=0.15]; \draw (4,5) circle [radius=0.15]; 
    \draw[fill] (4,2) circle [radius=0.15]; 
    \draw (0.5,5.5) circle [radius=0.3];
    \draw (1.5,4.5) circle [radius=0.3];
    \draw (2.5,3.5) circle [radius=0.3];
    \draw (3.5,2.5) circle [radius=0.3];
    \draw (4.5,1.5) circle [radius=0.3];
    \draw (5.5,0.5) circle [radius=0.3];    
   \end{tikzpicture}
  \end{center}
  \caption{The state $y\in I(D)$ is marked with the black circles}
  \label{2}
 \end{figure}
 We denote with $I(D)\subset S(D)$ the subset of grid states that contain the point $(0,\alpha)$ as in Figure \ref{2}.
 Every $y\in I(D)$ is the extension of a grid state $y'$ of the diagram $D'$ obtained by removing the first column and the last
 row from $D$. By the inductive hypothesis we have $M(y)=-1+M(y')\leq-1$.
 
 Now it easy to see that every other $z\in S(D)$ is obtained by a rectangle move from a $y\in I(D)$. Then, if $r$ is the 
 rectangle, we have
 $$M(z)-M(y)=1-2\cdot\#|r\cap\OO|+2\cdot\#|\text{Int}(r)\cap y|\:,$$
 but $\#|\text{Int}(r)\cap y|\leq\text{min}\{\pi_1(r),\pi_2(r)\}=\#|r\cap\OO|$ where $\pi_i(r)$ is the lenght of the edges
 of $r$.
 Hence $M(z)\leq 0$.
 \end{enumerate}
\end{proof}
In Lemma \ref{lemma:generator} we used that the grid diagram $D$ has all the $O$-markings aligned on a diagonal.
It is easy to see that if a link admits such a diagram then it is positive. On the other hand, it seems difficult for the
converse to be true.

\subsection{Torus links}
\label{subsection:torus}
We compute the $\tau$-invariant of every torus link.
Consider the grid diagram $D_{q,p}$ in Figure \ref{3},
representing the torus link $T_{q,p}$ with $q\leq p$ and all the components oriented
in the same direction.
\begin{figure}[ht]
  \begin{center}
   \begin{tikzpicture}[scale=0.5]
    \draw[help lines] (0,0) grid (7,7); 
    \draw[fill] (0,0) circle [radius=0.15]; \draw (0.5,6.5) circle [radius=0.3];
                                            \draw (0.5,6.5) circle [radius=0.2];
    \draw[fill] (0,7) circle [radius=0.15]; \draw (1.5,5.5) circle [radius=0.3];
    \draw[fill] (1,6) circle [radius=0.15]; \draw (2.5,4.5) circle [radius=0.3];
    \draw[fill] (2,5) circle [radius=0.15]; \draw (3.5,3.5) circle [radius=0.3];
    \draw[fill] (3,4) circle [radius=0.15]; \draw (4.5,2.5) circle [radius=0.3];
    \draw[fill] (4,3) circle [radius=0.15]; \draw (5.5,1.5) circle [radius=0.3];
    \draw[fill] (5,2) circle [radius=0.15]; \draw (6.5,0.5) circle [radius=0.3];
    \draw[fill] (6,1) circle [radius=0.15];
    \draw[fill] (7,0) circle [radius=0.15];
    \draw[fill] (7,7) circle [radius=0.15];
    \node at (0.5,2.5) {X}; 
    \node at (1.5,1.5) {X};
    \node at (2.5,0.5) {X};
    \node at (3.5,6.5) {X};
    \node at (4.5,5.5) {X};
    \node at (5.5,4.5) {X};
    \node at (6.5,3.5) {X};
    \draw[|-|] (7.5,0)--(7.5,3); \draw[-|] (7.5,3)--(7.5,7);
    \node[right] at (8,1.5) {$q$}; \node[right] at (8,5) {$p$};
   \end{tikzpicture}
  \end{center}
  \caption{$x$ is the grid state in the picture}
  \label{3}
\end{figure}
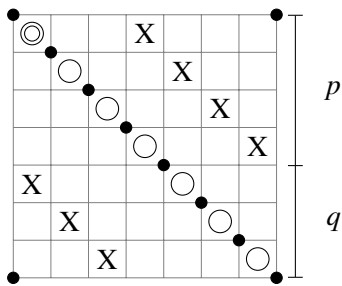
By Lemma \ref{lemma:generator} we know that $[x]$ is the only generator of $\widehat{\mathcal{GH}}_0(T_{q,p})$. If we 
denote with $n$ the number of components of $T_{q,p}$ then a simple 
computation gives
$$\begin{aligned}
   A(x)&=\frac{1}{2}\left(M(x)-M_{\X}(x)-\text{grd}(D_{q,p})+n\right)=\frac{1}{2}(-M_{\X}(x)-p-q+n)=\\
   &=\frac{1}{2}\left[2\sum_{i=1}^{q-1}i+q(p-q+1)-p-q+n\right]=\frac{1}{2}[q(q-1)+q(p-q+1)-p-q+n]=\\
   &=\frac{1}{2}(pq-p-q+n)=\frac{(p-1)(q-1)-1+n}{2}\:.
  \end{aligned}$$
Now we use Lemma \ref{lemma:generator} again and obtain that
\begin{equation*}
 \tau(T_{q,p})=\frac{(p-1)(q-1)-1+n}{2}\:.
\end{equation*}
Using the lower bound of Equation
\eqref{bound} gives a different way to compute the slice genus of a torus link respect to what we did
in \cite{Cavallo}:
$$g_4(T_{q,p})\geq\frac{(p-1)(q-1)}{2}-\frac{1-n}{2}+1-n=\frac{(p-1)(q-1)+1-n}{2}\:.$$
Since the Seifert algorithm applied to the standard diagram of $T_{q,p}$ gives the opposite inequality, we conclude
that $$g_4(T_{q,p})=\frac{(p-1)(q-1)+1-n}{2}\:\:\:\:\:\text{for any }q\leq p\:.$$ 

\subsection{Applications to Legendrian invariants}
We equip $S^3$ with its unique tight contact structure $\xi_{\text{st}}$, whose definition 
can be found in \cite{Geiges} Chapter 2.
We want to prove that Equation \eqref{tb_bound} holds in this case:
$$\text{tb}(\mathcal L)+|\text{rot}(\mathcal L)|\leq2\tau(L)-n\:.$$
We remark that, if $\mathcal D=\mathcal D_1\cup...\cup\mathcal D_n$
is a front projection of the Legendrian link $\mathcal L$ in the standard contact 3-sphere,
the Thurston-Bennequin and rotation number
of $\mathcal L$ are given by $$\text{tb}(\mathcal L)=\displaystyle\sum_{i=1}^n\text{tb}_i(\mathcal L)\:\:\:\:\:
\text{and}\:\:\:\:\:\text{rot}(\mathcal L)=\displaystyle\sum_{i=1}^n\text{rot}_i(\mathcal L)$$ where
$$\text{tb}_i(\mathcal L)=w(\mathcal D_i)+\text{lk}\left(\mathcal D_i,\mathcal D\setminus\mathcal D_i\right)-\frac{1}{2}
\#|\text{cusps in }\mathcal D_i|$$ and
$$\text{rot}_i(\mathcal L)=\frac{1}{2}\bigg(\#|\text{downward cusps in }\mathcal D_i|-\#|\text{upward cusps in } \mathcal D_i|
\bigg)\:;$$ here we denote with $w$ the writhe of a link diagram. We could simply say that $$\text{tb}(\mathcal L)=
w(\mathcal D)-\frac{1}{2}\#|\text{cusps in }\mathcal D|$$ and 
$$\text{rot}(\mathcal L)=\frac{1}{2}\bigg(\#|\text{downward cusps in }\mathcal D|
-\#|\text{upward cusps in }\mathcal D|\bigg)\:,$$
but we need the previous definition in the following proof.
\begin{proof}[Proof of Proposition \ref{prop:three}]
 If $\mathcal L$ is a Legendrian link then, from Chapter 12 in \cite{Book}, we know that $\mathcal L$ can be represented by
 a grid diagram $D$ of the link $L^*$ (the mirror of $L$). This diagram $D$ is such that 
 $$\frac{\text{tb}_i(\mathcal L)-\text{rot}_i(\mathcal L)+1}{2}=A_i(x^+)\:\:\:\:\:
 \frac{\text{tb}_i(\mathcal L)+\text{rot}_i(\mathcal L)+1}{2}=A_i(x^-)$$
 $$\text{tb}(\mathcal L)-\text{rot}(\mathcal L)+1=M(x^+)\:\:\:\:\:
 \text{tb}(\mathcal L)+\text{rot}(\mathcal L)+1=M(x^-)\:,$$
 where $x^{\pm}$ are the grid states in $D$ obtained by taking a point in the
 northeast (southwest for $x^-$) corner of every square decorated with an $X\in\X$. Moreover, $A_i$ is defined as follows:
 $$A_i(x)=\mathcal J\bigg(x-\frac{1}{2}(\X+\OO),(\X_i-\OO_i)\bigg)-\frac{\text{grd}(D)_i-1}{2}\:\:\:\:\:\text{for any }
 x\in S(D)$$
 with $\OO_i\subset\OO$ and $\X_i\subset\X$ the markings on the $i$-th component of $L^*$ and $\text{grd}(D)_i$ 
 the number of elements in $\OO_i$.

 From Lemma 8.4.7 and Theorems 12.3.2 and
 12.7.5 in \cite{Book} we have that $x^{\pm}$ represent non-torsion elements in the homology group $cGH^-(L^*)$;
 in fact these classes are the Legendrian grid invariants $\lambda^{\pm}(\mathcal L)$.
 Since $A(x)=\displaystyle\sum_{i=1}^nA_i(x)$ for every grid state $x$, we have that 
 $M\left(x^{\pm}\right)=2A\left(x^{\pm}\right)+1-n$. Therefore, Corollary \ref{cor:twice}
 implies that $A\left(x^{\pm}\right)\leq-\tau^*(L^*)$. Combining the latter claim with Corollary \ref{cor:mirror}, that 
 gives $\tau^*(L^*)=-\tau(L)$, we have 
 $$\frac{\text{tb}(\mathcal L)\mp\text{rot}(\mathcal L)+n}{2}=A(x^{\pm})\leq\tau(L)$$
 that is precisely Equation \eqref{tb_bound}.
\end{proof}
From Equation \eqref{tb_bound}, together with Equation \eqref{bound}, we obtain the lower bound for the slice genus of 
Equation \eqref{tb_slice}:
$$\text{tb}(\mathcal L)+|\text{rot}(\mathcal L)|\leq2g_4(L)+n-2\:.$$
This bound is sharp for positive torus links, but here we show that there are 
other links for which this happens. 

In Figure \ref{L9n19} we have a front projection $\mathcal D$ of a Legendrian two 
component link $\mathcal L$. The link type of $\mathcal L$ is the link $L9^n_{19}$.
A simple computation gives $\text{tb}(\mathcal L)=6$ and $\text{rot}(\mathcal L)=0$, therefore Equation \eqref{tb_slice}
says that $g_4(L9^n_{19})\geq 3$. Since it is easy to see that the link represented by $\mathcal D$ can be 
unlinked by changing the four crossings highlighted in the picture, we have $g_4(L9^n_{19})\leq
u(L9^n_{19})-1\leq 3$ and then we conclude that $g_4(L9^n_{19})=3$.
\begin{figure}[ht]
        \centering
        \def\svgwidth{12cm}
        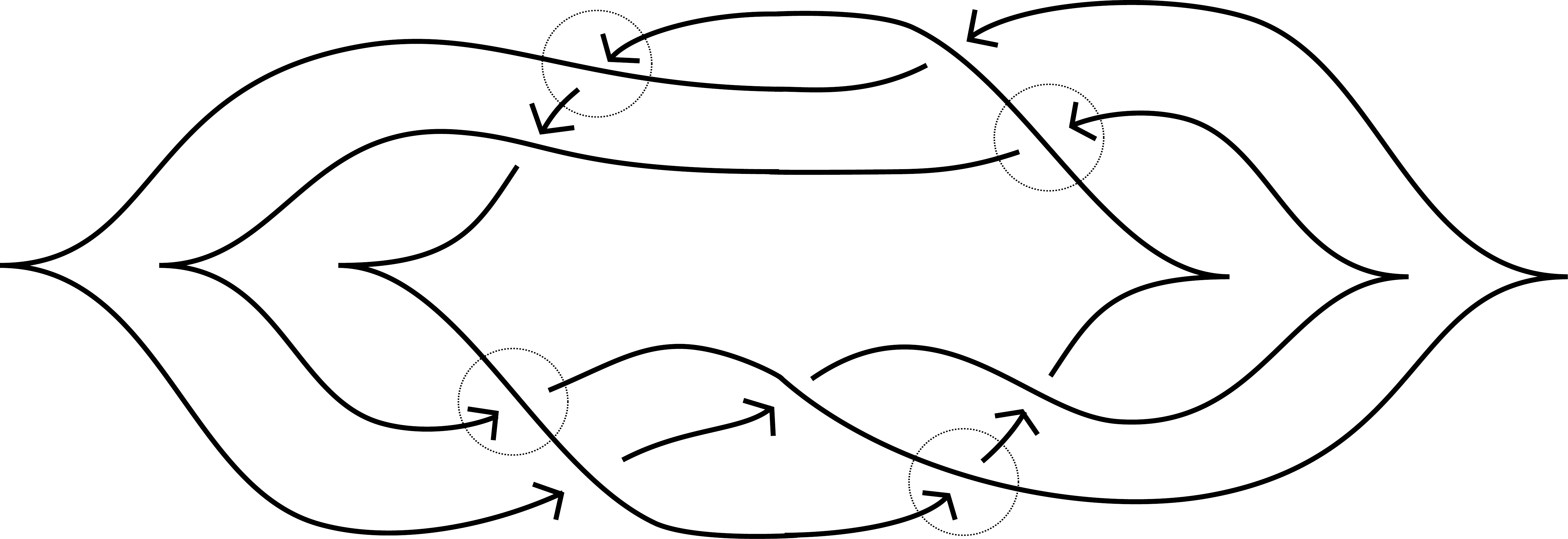   
        \caption{A diagram of the link $L9^n_{19}$}
        \label{L9n19}
\end{figure}

From Equation \eqref{tb_bound} we also have the upper bound for the maximal Thurston-Bennequin number of 
Proposition \ref{prop:max}:
$$\text{TB}(L)\leq 2\tau(L)-n$$
and its refinement for quasi-alternating links given by Corollary \ref{cor:tb}.
Although this bound is much less powerful than the Kauffmann or the HOMFLY polynomial, we can still get some interesting
conclusions. 

Consider a Legendrian link $\mathcal L$ such that each component $\mathcal L_i$ is algebraically unknotted.
Then $\text{tb}_i(\mathcal L)=\text{tb}(\mathcal L_i)$ and so $\text{tb}(\mathcal L)$ is precisely the sum of the 
Thurston-Bennequin numbers of its components.
For example this happens for the Borromean link $B$, whose components $B_i$ are three 
(algebraically unknotted) unknots. It was shown in \cite{Mohnke} that there
is no Legendrian representation of $B$, 
where the Thurston-Bennequin number of each component is -1; in fact we have $\text{TB}(B)=-4$, while 
$\text{TB}(\bigcirc)=-1$. In particular, this means  
that the difference between $\text{TB}(B)$ and the sum of $\text{TB}(B_i)$ is -1.

We prove Proposition \ref{prop:last}, where we give a family of two components links $L^k$ such that
the components of $L^k$ are two unknots with $\text{lk}(L^k_1,L^k_2)=0$ and
the difference between $\text{TB}(L^k)$ and the sum of $\text{TB}(L^k_i)$ is actually arbitrarily small,
improving the latter result for $B$.
The links $L^k$ are shown in Figure \ref{L9a40}.
\begin{proof}[Proof of Proposition \ref{prop:last}]
Since for every $k\geq 0$ the link $L^k$ is non-split alternating, we can easily compute the signature that is equal to
$3+2k$. Now we apply Corollary \ref{cor:tb} and we obtain that $\text{TB}(L^k)\leq -4-2k$.
\end{proof}

\end{document}

%% file: L9a40.pdf_tex
%% Creator: Inkscape inkscape 0.48.3.1, www.inkscape.org
%% Accompanies image file 'L9a40.pdf' (pdf, eps, ps)
\begingroup%
  \makeatletter%
  \providecommand\color[2][]{%
    \errmessage{(Inkscape) Color is used for the text in Inkscape, but the package 'color.sty' is not loaded}%
    \renewcommand\color[2][]{}%
  }%
  \providecommand\transparent[1]{%
    \errmessage{(Inkscape) Transparency is used (non-zero) for the text in Inkscape, but the package 'transparent.sty' is not loaded}%
    \renewcommand\transparent[1]{}%
  }%
  \providecommand\rotatebox[2]{#2}%
  \ifx\svgwidth\undefined%
    \setlength{\unitlength}{2145.61318359bp}%
    \ifx\svgscale\undefined%
      \relax%
    \else%
      \setlength{\unitlength}{\unitlength * \real{\svgscale}}%
    \fi%
  \else%
    \setlength{\unitlength}{\svgwidth}%
  \fi%
  \global\let\svgwidth\undefined%
  \global\let\svgscale\undefined%
  \makeatother%
  \begin{picture}(1,0.60566195)%
    \put(0,0){\includegraphics[width=\unitlength]{L9a40.pdf}}%
    \put(0.02998148,0.26511144){\color[rgb]{0,0,0}\makebox(0,0)[lb]{\smash{$-2k$}}}%
    \put(0.87902278,0.26298085){\color[rgb]{0,0,0}\makebox(0,0)[lb]{\smash{$+2k$}}}%
  \end{picture}%
\endgroup%

%% file: Identity.pdf_tex
%% Creator: Inkscape inkscape 0.48.3.1, www.inkscape.org
%% Accompanies image file 'Identity.pdf' (pdf, eps, ps)
\begingroup%
  \makeatletter%
  \providecommand\color[2][]{%
    \errmessage{(Inkscape) Color is used for the text in Inkscape, but the package 'color.sty' is not loaded}%
    \renewcommand\color[2][]{}%
  }%
  \providecommand\transparent[1]{%
    \errmessage{(Inkscape) Transparency is used (non-zero) for the text in Inkscape, but the package 'transparent.sty' is not loaded}%
    \renewcommand\transparent[1]{}%
  }%
  \providecommand\rotatebox[2]{#2}%
  \ifx\svgwidth\undefined%
    \setlength{\unitlength}{1444.59487305bp}%
    \ifx\svgscale\undefined%
      \relax%
    \else%
      \setlength{\unitlength}{\unitlength * \real{\svgscale}}%
    \fi%
  \else%
    \setlength{\unitlength}{\svgwidth}%
  \fi%
  \global\let\svgwidth\undefined%
  \global\let\svgscale\undefined%
  \makeatother%
  \begin{picture}(1,0.64939363)%
    \put(0,0){\includegraphics[width=\unitlength]{Identity.pdf}}%
    \put(-0.00459688,0.53599524){\color[rgb]{0,0,0}\makebox(0,0)[lb]{\smash{$L_1$}}}%
    \put(0.84032612,0.53441299){\color[rgb]{0,0,0}\makebox(0,0)[lb]{\smash{$L_2$}}}%
  \end{picture}%
\endgroup%

%% file: Merge.pdf_tex
%% Creator: Inkscape inkscape 0.48.3.1, www.inkscape.org
%% Accompanies image file 'Merge.pdf' (pdf, eps, ps)
\begingroup%
  \makeatletter%
  \providecommand\color[2][]{%
    \errmessage{(Inkscape) Color is used for the text in Inkscape, but the package 'color.sty' is not loaded}%
    \renewcommand\color[2][]{}%
  }%
  \providecommand\transparent[1]{%
    \errmessage{(Inkscape) Transparency is used (non-zero) for the text in Inkscape, but the package 'transparent.sty' is not loaded}%
    \renewcommand\transparent[1]{}%
  }%
  \providecommand\rotatebox[2]{#2}%
  \ifx\svgwidth\undefined%
    \setlength{\unitlength}{1389.73771973bp}%
    \ifx\svgscale\undefined%
      \relax%
    \else%
      \setlength{\unitlength}{\unitlength * \real{\svgscale}}%
    \fi%
  \else%
    \setlength{\unitlength}{\svgwidth}%
  \fi%
  \global\let\svgwidth\undefined%
  \global\let\svgscale\undefined%
  \makeatother%
  \begin{picture}(1,0.79095385)%
    \put(0,0){\includegraphics[width=\unitlength]{Merge.pdf}}%
    \put(-0.00477833,0.54764668){\color[rgb]{0,0,0}\makebox(0,0)[lb]{\smash{$L_1$}}}%
    \put(0.83402331,0.54929139){\color[rgb]{0,0,0}\makebox(0,0)[lb]{\smash{$L_2$}}}%
  \end{picture}%
\endgroup%

%% file: Split.pdf_tex
%% Creator: Inkscape inkscape 0.48.3.1, www.inkscape.org
%% Accompanies image file 'Split.pdf' (pdf, eps, ps)
\begingroup%
  \makeatletter%
  \providecommand\color[2][]{%
    \errmessage{(Inkscape) Color is used for the text in Inkscape, but the package 'color.sty' is not loaded}%
    \renewcommand\color[2][]{}%
  }%
  \providecommand\transparent[1]{%
    \errmessage{(Inkscape) Transparency is used (non-zero) for the text in Inkscape, but the package 'transparent.sty' is not loaded}%
    \renewcommand\transparent[1]{}%
  }%
  \providecommand\rotatebox[2]{#2}%
  \ifx\svgwidth\undefined%
    \setlength{\unitlength}{1392.0234375bp}%
    \ifx\svgscale\undefined%
      \relax%
    \else%
      \setlength{\unitlength}{\unitlength * \real{\svgscale}}%
    \fi%
  \else%
    \setlength{\unitlength}{\svgwidth}%
  \fi%
  \global\let\svgwidth\undefined%
  \global\let\svgscale\undefined%
  \makeatother%
  \begin{picture}(1,0.79128009)%
    \put(0,0){\includegraphics[width=\unitlength]{Split.pdf}}%
    \put(-0.00477048,0.55001445){\color[rgb]{0,0,0}\makebox(0,0)[lb]{\smash{$L_1$}}}%
    \put(0.83429585,0.55658248){\color[rgb]{0,0,0}\makebox(0,0)[lb]{\smash{$L_2$}}}%
  \end{picture}%
\endgroup%

%% file: Torus.pdf_tex
%% Creator: Inkscape inkscape 0.48.3.1, www.inkscape.org
%% Accompanies image file 'Torus.pdf' (pdf, eps, ps)
\begingroup%
  \makeatletter%
  \providecommand\color[2][]{%
    \errmessage{(Inkscape) Color is used for the text in Inkscape, but the package 'color.sty' is not loaded}%
    \renewcommand\color[2][]{}%
  }%
  \providecommand\transparent[1]{%
    \errmessage{(Inkscape) Transparency is used (non-zero) for the text in Inkscape, but the package 'transparent.sty' is not loaded}%
    \renewcommand\transparent[1]{}%
  }%
  \providecommand\rotatebox[2]{#2}%
  \ifx\svgwidth\undefined%
    \setlength{\unitlength}{1949.73769531bp}%
    \ifx\svgscale\undefined%
      \relax%
    \else%
      \setlength{\unitlength}{\unitlength * \real{\svgscale}}%
    \fi%
  \else%
    \setlength{\unitlength}{\svgwidth}%
  \fi%
  \global\let\svgwidth\undefined%
  \global\let\svgscale\undefined%
  \makeatother%
  \begin{picture}(1,0.5442114)%
    \put(0,0){\includegraphics[width=\unitlength]{Torus.pdf}}%
    \put(-0.00340591,0.36884807){\color[rgb]{0,0,0}\makebox(0,0)[lb]{\smash{$L_1$}}}%
    \put(0.88169482,0.36533111){\color[rgb]{0,0,0}\makebox(0,0)[lb]{\smash{$L_2$}}}%
  \end{picture}%
\endgroup%

%% file: Birth.pdf_tex
%% Creator: Inkscape inkscape 0.48.3.1, www.inkscape.org
%% Accompanies image file 'Birth.pdf' (pdf, eps, ps)
\begingroup%
  \makeatletter%
  \providecommand\color[2][]{%
    \errmessage{(Inkscape) Color is used for the text in Inkscape, but the package 'color.sty' is not loaded}%
    \renewcommand\color[2][]{}%
  }%
  \providecommand\transparent[1]{%
    \errmessage{(Inkscape) Transparency is used (non-zero) for the text in Inkscape, but the package 'transparent.sty' is not loaded}%
    \renewcommand\transparent[1]{}%
  }%
  \providecommand\rotatebox[2]{#2}%
  \ifx\svgwidth\undefined%
    \setlength{\unitlength}{1385.16630859bp}%
    \ifx\svgscale\undefined%
      \relax%
    \else%
      \setlength{\unitlength}{\unitlength * \real{\svgscale}}%
    \fi%
  \else%
    \setlength{\unitlength}{\svgwidth}%
  \fi%
  \global\let\svgwidth\undefined%
  \global\let\svgscale\undefined%
  \makeatother%
  \begin{picture}(1,0.63642808)%
    \put(0,0){\includegraphics[width=\unitlength]{Birth.pdf}}%
    \put(0.83347555,0.53527585){\color[rgb]{0,0,0}\makebox(0,0)[lb]{\smash{$L_2$}}}%
    \put(-0.0047941,0.07323748){\color[rgb]{0,0,0}\makebox(0,0)[lb]{\smash{$L_1$}}}%
  \end{picture}%
\endgroup%

%% file: Birth1.pdf_tex
%% Creator: Inkscape inkscape 0.48.3.1, www.inkscape.org
%% Accompanies image file 'Birth1.pdf' (pdf, eps, ps)
\begingroup%
  \makeatletter%
  \providecommand\color[2][]{%
    \errmessage{(Inkscape) Color is used for the text in Inkscape, but the package 'color.sty' is not loaded}%
    \renewcommand\color[2][]{}%
  }%
  \providecommand\transparent[1]{%
    \errmessage{(Inkscape) Transparency is used (non-zero) for the text in Inkscape, but the package 'transparent.sty' is not loaded}%
    \renewcommand\transparent[1]{}%
  }%
  \providecommand\rotatebox[2]{#2}%
  \ifx\svgwidth\undefined%
    \setlength{\unitlength}{2009.16625977bp}%
    \ifx\svgscale\undefined%
      \relax%
    \else%
      \setlength{\unitlength}{\unitlength * \real{\svgscale}}%
    \fi%
  \else%
    \setlength{\unitlength}{\svgwidth}%
  \fi%
  \global\let\svgwidth\undefined%
  \global\let\svgscale\undefined%
  \makeatother%
  \begin{picture}(1,0.40273091)%
    \put(0,0){\includegraphics[width=\unitlength]{Birth1.pdf}}%
    \put(0.88519414,0.19228213){\color[rgb]{0,0,0}\makebox(0,0)[lb]{\smash{$L_2$}}}%
    \put(-0.00330516,0.0466638){\color[rgb]{0,0,0}\makebox(0,0)[lb]{\smash{$L_1$}}}%
  \end{picture}%
\endgroup%

%% file: Snail.pdf_tex
%% Creator: Inkscape inkscape 0.48.3.1, www.inkscape.org
%% Accompanies image file 'Snail.pdf' (pdf, eps, ps)
\begingroup%
  \makeatletter%
  \providecommand\color[2][]{%
    \errmessage{(Inkscape) Color is used for the text in Inkscape, but the package 'color.sty' is not loaded}%
    \renewcommand\color[2][]{}%
  }%
  \providecommand\transparent[1]{%
    \errmessage{(Inkscape) Transparency is used (non-zero) for the text in Inkscape, but the package 'transparent.sty' is not loaded}%
    \renewcommand\transparent[1]{}%
  }%
  \providecommand\rotatebox[2]{#2}%
  \ifx\svgwidth\undefined%
    \setlength{\unitlength}{3268.0515625bp}%
    \ifx\svgscale\undefined%
      \relax%
    \else%
      \setlength{\unitlength}{\unitlength * \real{\svgscale}}%
    \fi%
  \else%
    \setlength{\unitlength}{\svgwidth}%
  \fi%
  \global\let\svgwidth\undefined%
  \global\let\svgscale\undefined%
  \makeatother%
  \begin{picture}(1,0.21166168)%
    \put(0,0){\includegraphics[width=\unitlength]{Snail.pdf}}%
  \end{picture}%
\endgroup%

%% file: Death.pdf_tex
%% Creator: Inkscape inkscape 0.48.3.1, www.inkscape.org
%% Accompanies image file 'Death.pdf' (pdf, eps, ps)
\begingroup%
  \makeatletter%
  \providecommand\color[2][]{%
    \errmessage{(Inkscape) Color is used for the text in Inkscape, but the package 'color.sty' is not loaded}%
    \renewcommand\color[2][]{}%
  }%
  \providecommand\transparent[1]{%
    \errmessage{(Inkscape) Transparency is used (non-zero) for the text in Inkscape, but the package 'transparent.sty' is not loaded}%
    \renewcommand\transparent[1]{}%
  }%
  \providecommand\rotatebox[2]{#2}%
  \ifx\svgwidth\undefined%
    \setlength{\unitlength}{2002.30859375bp}%
    \ifx\svgscale\undefined%
      \relax%
    \else%
      \setlength{\unitlength}{\unitlength * \real{\svgscale}}%
    \fi%
  \else%
    \setlength{\unitlength}{\svgwidth}%
  \fi%
  \global\let\svgwidth\undefined%
  \global\let\svgscale\undefined%
  \makeatother%
  \begin{picture}(1,0.40396291)%
    \put(0,0){\includegraphics[width=\unitlength]{Death.pdf}}%
    \put(0.88480094,0.05444762){\color[rgb]{0,0,0}\makebox(0,0)[lb]{\smash{$L_2$}}}%
    \put(-0.00331648,0.18914929){\color[rgb]{0,0,0}\makebox(0,0)[lb]{\smash{$L_1$}}}%
  \end{picture}%
\endgroup%

%% file: Big1.pdf_tex
%% Creator: Inkscape inkscape 0.48.3.1, www.inkscape.org
%% Accompanies image file 'Big1.pdf' (pdf, eps, ps)
\begingroup%
  \makeatletter%
  \providecommand\color[2][]{%
    \errmessage{(Inkscape) Color is used for the text in Inkscape, but the package 'color.sty' is not loaded}%
    \renewcommand\color[2][]{}%
  }%
  \providecommand\transparent[1]{%
    \errmessage{(Inkscape) Transparency is used (non-zero) for the text in Inkscape, but the package 'transparent.sty' is not loaded}%
    \renewcommand\transparent[1]{}%
  }%
  \providecommand\rotatebox[2]{#2}%
  \ifx\svgwidth\undefined%
    \setlength{\unitlength}{2728.83783433bp}%
    \ifx\svgscale\undefined%
      \relax%
    \else%
      \setlength{\unitlength}{\unitlength * \real{\svgscale}}%
    \fi%
  \else%
    \setlength{\unitlength}{\svgwidth}%
  \fi%
  \global\let\svgwidth\undefined%
  \global\let\svgscale\undefined%
  \makeatother%
  \begin{picture}(1,0.47731144)%
    \put(0,0){\includegraphics[width=\unitlength]{Big1.pdf}}%
    \put(0.21355068,0.00775155){\color[rgb]{0,0,0}\makebox(0,0)[lb]{\smash{$t_1$}}}%
    \put(0.60890477,0.006914){\color[rgb]{0,0,0}\makebox(0,0)[lb]{\smash{$t_2$}}}%
    \put(0.00247181,0.39305428){\color[rgb]{0,0,0}\makebox(0,0)[lb]{\smash{$L_1$}}}%
    \put(0.91547169,0.45503776){\color[rgb]{0,0,0}\makebox(0,0)[lb]{\smash{$L_2$}}}%
    \put(0.00163418,0.0094268){\color[rgb]{0,0,0}\makebox(0,0)[lb]{\smash{0}}}%
    \put(0.92049736,0.00775159){\color[rgb]{0,0,0}\makebox(0,0)[lb]{\smash{1}}}%
  \end{picture}%
\endgroup%

%% file: L9n19.pdf_tex
%% Creator: Inkscape inkscape 0.48.3.1, www.inkscape.org
%% Accompanies image file 'L9n19.pdf' (pdf, eps, ps)
\begingroup%
  \makeatletter%
  \providecommand\color[2][]{%
    \errmessage{(Inkscape) Color is used for the text in Inkscape, but the package 'color.sty' is not loaded}%
    \renewcommand\color[2][]{}%
  }%
  \providecommand\transparent[1]{%
    \errmessage{(Inkscape) Transparency is used (non-zero) for the text in Inkscape, but the package 'transparent.sty' is not loaded}%
    \renewcommand\transparent[1]{}%
  }%
  \providecommand\rotatebox[2]{#2}%
  \ifx\svgwidth\undefined%
    \setlength{\unitlength}{2521.15bp}%
    \ifx\svgscale\undefined%
      \relax%
    \else%
      \setlength{\unitlength}{\unitlength * \real{\svgscale}}%
    \fi%
  \else%
    \setlength{\unitlength}{\svgwidth}%
  \fi%
  \global\let\svgwidth\undefined%
  \global\let\svgscale\undefined%
  \makeatother%
  \begin{picture}(1,0.34396662)%
    \put(0,0){\includegraphics[width=\unitlength]{L9n19.pdf}}%
  \end{picture}%
\endgroup%